\documentclass[a4paper,11pt]{article}
\usepackage{amsmath}
\usepackage{amsfonts}
\usepackage{supertabular}
\usepackage[latin9]{inputenc}
\usepackage[english]{varioref}
\usepackage{dcolumn}
\usepackage[height=22cm , width = 16cm , top = 4cm , left = 3cm, a4paper]{geometry}
\usepackage[a4paper]{geometry}
\usepackage[final]{graphicx}
\usepackage{epsfig}
\usepackage{pstricks}
\usepackage{psfrag}
\usepackage{rotating}
\usepackage{supertabular}
\usepackage{booktabs}
\usepackage{delarray}
\usepackage{rotating}
\usepackage[numbers]{natbib}
\usepackage{subfigure}
\usepackage{nextpage}
\usepackage{layout}
\usepackage{amsthm}
\usepackage{dsfont}
\usepackage{hyperref}
\usepackage[hypcap]{caption}
\usepackage[autostyle]{csquotes}
\usepackage{color}
\usepackage{calrsfs}

\numberwithin{equation}{section}
{                     
{                     
{                       
{                       

\theoremstyle{plain}
\newtheorem{thm}{Theorem}[section]

\newtheorem{lem}[thm]{Lemma}

\parindent0cm
\newcommand{\enter}{\bigskip}

\date{}
\begin{document}
 \author{{Prasanta Kumar Barik and Ankik Kumar Giri}\vspace{.2cm}\footnote{Corresponding author. Tel +91-1332-284818 (O);  Fax: +91-1332-273560  \newline{\it{${}$ \hspace{.3cm} Email address: }}ankikgiri.fma@iitr.ac.}\\
\footnotesize  \small{ \textit{Department of Mathematics, Indian Institute of Technology Roorkee,  Roorkee-247667, Uttarakhand,}}\\ \small{\textit{India}}
  }

\title{The continuous coagulation and nonlinear multiple fragmentation equation}
\date{January 10, 2018}

\maketitle


\begin{quote}
{\small {\em\bf Abstract} The present paper deals with the existence and uniqueness of global classical solutions to the continuous coagulation and nonlinear multiple fragmentation equation for large classes of unbounded coagulation, collision and breakup kernels. In addition, it is shown that solutions are mass conserving. The coagulation and breakup kernels may have singularities on both the co-ordinate axes whereas the collision kernel grows up to bilinearity.  \enter
}
\end{quote}
{\rm \bf MSC 2010:} Primary: 45K05; 47J35; Secondary: 34K30; 34G20;\\

{\bf Keywords:}  Coagulation; Nonlinear multiple fragmentation; Collision kernel; Global solution; Existence; Uniqueness; Mass conservation.


\section{Introduction}\label{existintroduction1}

We study the general continuous coagulation and nonlinear multiple fragmentation equation which describes the evolution of the size distribution function
 $g=g(x,t) \geq 0$ of particles of mass $x \in \mathds{R}_{>0}:= (0, \infty)$ at time $t \geq 0$ and reads \cite{Camejo:2014I, Kostoglou:2000, Barik:2017III, Barik:2017II}
\begin{align}\label{ccfe}
\frac {\partial g(x,t)}{\partial t} =&\frac{1}{2} \int_0^{x}K(x-y,y)g(x-y,t)g(y,t)dy - \int_{0}^{\infty} K(x,y)g(x,t)g(y,t)dy\nonumber\\
&+\int_{x}^{\infty}\int_{0}^{\infty}B(x|y;z)C(y,z)g(y,t)g(z,t)dzdy-\int_{0}^{\infty} C(x,y)g(x,t)g(y,t)dy
\end{align}
with
\begin{align}\label{in1}
g(x,0) = g(x, 0)\geq 0\  \mbox{a.e.}.
\end{align}
The multiple fragmentation process considered in (\ref{ccfe}) is nonlinear which occurs due to the collision between pairs of particles. More details on the nonlinear fragmentation can be obtained in \cite{Cheng:1990, Cheng:1988, Matthieu:2007, Kostoglou:2000, Laurencot:2001}.\\

In equation (\ref{ccfe}), $K$ and $C$ describe the coagulation and collision kernels, respectively, which are both nonnegative and symmetric functions $K(x, y) = K(y, x)$  and $C(x, y) = C(y, x)$ for all $(x, y) \in \mathds{R}_{>0} \times \mathds{R}_{>0} $. The coagulation kernel $K(x, y)$ denotes the rate at which two particles with respective masses $x$ and $y$ merge to form particles of mass $x+y$ whereas $C(x, y)$  represents the rate of collision between particles of masses $x$ and $y$. In general, the collision kernels have similar properties like coagulation kernels, which are described in \cite{Barik:2017I, Camejo:2014I, Giri:2012, Stewart:1989}. Next, the breakup kernel $B(x|y;z)$ is the daughter distribution function describing the probability that the fragmentation of a particle with mass $y$ produces a particle with mass $x \in (0, y)$ after the collision with particles of mass $z$. The mass conservation property during the fragmentation events demands
\begin{align}\label{mass1}
\int_{0}^{y}xB(x|y;z)dx = y\ \ \text{for all}\ \ y\in \mathds{R}_{>0}.
\end{align}
In addition, it is important to mention that the total number of fragments $N(y)$ resulting from the breakage of a particle of mass $y > 0$ can be defined as
\begin{align}\label{N1}
N(y):= \int_{0}^{y}B(x|y;z)dx< \infty, \ \ \text{for all}\ \ y>0,\ \ \ \ \ B(x|y;z)=0\ \ \text{for}\ \ x> y,
\end{align}
where $\sup_{y \in \mathds{R}_{>0}}N(y) =N$. The breakup kernel has similar behaviour like breakage function in linear fragmentation equation.\\

In the right-hand side of (\ref{ccfe}), the first term represents the appearance of particles of mass $x$ due to binary coalescence of smaller ones and the second one exhibits the disappearance of particles of mass $x$ due to coagulation events. The third term describes the creation of particles mass $x$ due to the collision of particles of masses $y$ and $z$ while the fourth term accounts the depletion of particles of mass $x$ due to the collision with other particles during the nonlinear fragmentation events.\\

Moreover, throughout this paper, we need to have some information on the time evolution of moments of solutions to the coagulation and nonlinear fragmentation equation (\ref{ccfe}). For a function $g : \mathds{R}_{>0} \rightarrow \mathds{R} $, we define the following explicit form of moments as

\begin{eqnarray}\label{moment}
\mathcal{M}_\xi (t) := \int_0^{\infty} x^{\xi} g(x, t)dx,\ \ \text{for}\ \xi \in \mathds{R}.
\end{eqnarray}
 The zeroth $(\xi =0)$ and first $(\xi =1)$ moments, $\mathcal{M}_0 (t)$ and $\mathcal{M}_1 (t)$, respectively, represent the total number of particles and total mass of particles. In coagulation events, the zeroth moment $\mathcal{M}_0 (t)$ is a  decreasing function while in fragmentation process, it is an increasing function. However, $\mathcal{M}_1 (t)$ may or may not be constant during coagulation and nonlinear fragmentation processes that depends on the nature of coagulation and collision rates. Negative moments are also important to mention here for including a few very interesting singular coagulation kernels (for e.g. Smoluchowski kernel for Brownian motion and Granulation kernel) in the theory of existence and uniqueness of solutions to (\ref{ccfe})--(\ref{in1}). These moments have been considered in many articles, see \cite{Canizo:2011, Escobedo:2004, Escobedo:2006, Fournier:2005, Norris:1999}.\\

The main novelty of this article is to show the existence, uniqueness and mass-conserving property of global solutions to (\ref{ccfe})--(\ref{in1}) for singular coagulation and breakup kernels alongwith the quadratic growth on collision kernels. The existence and uniqueness of solutions to the continuous coagulation with linear fragmentation processes have been discussed in many articles by using various techniques for different growth conditions on coagulation and fragmentation rates, see \cite{Banasiak:2012I, Banasiak:2012II, Banasiak:2011, Banasiak:2013, Camejo:2014II, Escobedo:2004, Giri:2013, Lamb:2004, McLaughlin:1997I, McLaughlin:1997II, Stewart:1989, Stewart:1990}. However, the nonlinear fragmentation equation has not been extensively addressed. Though there are a few articles in which analytical, scaling solutions and their asymptotic behaviour have been investigated, see \cite{Cheng:1990, Cheng:1988, Matthieu:2007, Jianhong:2008, Kostoglou:2000}. In \cite{Laurencot:2001}, authors have discussed the existence and uniqueness of weak solutions to the discrete non-linear fragmentation equation by using Weak $L^1$ compactness method. In addition, they have also investigated the asymptotic behaviour, gelation and mass conservation property of solutions. Later, in \cite{Jianhong:2008}, the analytical solutions to the discrete coagulation and nonlinear binary fragmentation equation have been studied. Recently, in \cite{Barik:2017III, Barik:2017II}, authors have discussed the existence and uniqueness of weak solutions to (\ref{ccfe})--(\ref{in1}) by taking both non-singular and singular coagulation kernels. However, the classes of collision kernels considered in \cite{Barik:2017III, Barik:2017II} are smaller than the one taken in the present work. Moreover, in the present work, we look for global classical solutions to (\ref{ccfe})--(\ref{in1}). Several researchers have already discussed the existence of global classical solutions to the coagulation and linear fragmentation equations through various techniques, see \cite{Banasiak:2012I, Banasiak:2012II, Banasiak:2011, Banasiak:2013, Dubovskii:1996, Galkin:1986, Lamb:2004, McLaughlin:1997I, McLaughlin:1997II, Saha:2014}. In \cite{Banasiak:2012I, Banasiak:2012II, Banasiak:2011, Banasiak:2013, Lamb:2004, McLaughlin:1997I, McLaughlin:1997II}, authors have discussed the existence of global classical solutions for the coagulation and linear fragmentation equations by using semigroup technique whereas in \cite{Dubovskii:1996, Galkin:1986, Saha:2014}, a different approach, introduced by Galkin and Dubovskii, is used to show the existence of global classical solutions which relies on a compactness argument. In \cite{Galkin:1986}, they have discussed the existence and uniqueness of global solutions to pure coagulation equation by taking an account of unbounded coagulation kernels. Later, in 1996, Dubovskii and Stewart \cite{Dubovskii:1996} have extended this result for the continuous coagulation and linear binary fragmentation equations with unbounded coagulation and fragmentation rates. Recently, Saha and Kumar \cite{Saha:2014} have extended the work of Dubovskii and Stewart \cite{Dubovskii:1996} for coagulation and linear multiple fragmentation equation by including singular coagulation rates. Best to our knowledge, this is the first attempt to address the existence, uniqueness and mass-conservation of global solutions for the continuous coagulation and nonlinear multiple fragmentation equation (\ref{ccfe})--(\ref{in1}). The main motivation of this work comes from \cite{Barik:2017III, Barik:2017II, Dubovskii:1996, Saha:2014}.\\

Let us workout the paper as per the following plans: In the next section, we state some function spaces and assumptions on coagulation, collision and breakup rates. In addition, the existence, uniqueness and mass-conservation results are also stated.  In Section $3$, the existence of global solutions is established by using compactness argument. Further, the mass-conserving property of solutions is also shown at the end of this section. Finally, the uniqueness result is investigated in the last section.

\section{Function spaces and Assumptions}
Fix $T>0$ and let us define the following abstract spaces. Let $\Xi$ be the strip defined as
\begin{eqnarray*}
\Xi := \{(x,t): x\in \mathbb{R}_{>0},\   0 \leq t \leq T \}
\end{eqnarray*}
and $\Xi(\lambda_1, \lambda_2; T)$ denotes the following closed rectangle
\begin{eqnarray*}
\Xi(\lambda_1, \lambda_2; T) := \{(x,t): x\in [\lambda_1, \lambda_2], 0 \leq t \leq T \},
\end{eqnarray*}
where $0< \lambda_1 \leq x \leq \lambda_2$. Let $\Lambda_{\sigma_1, \sigma_2}(T)$ be the space of all continuous functions $g$ with bounded norms defined by
\begin{eqnarray*}
\|g\|_{\sigma_1, \sigma_2}:=\sup_{0 \leq t \leq T} \int_0^{\infty} \bigg(x^{\sigma_1}+\frac{1}{x^{\sigma_2}} \bigg)|g(x,t)|dx,\ \  \sigma_1\geq 1\ \text{and}\ 0<\sigma_2<1
\end{eqnarray*}
and
\begin{eqnarray*}
\Lambda^{+}_{\sigma_1, \sigma_2}(T):= \{ g \in \Lambda_{\sigma_1, \sigma_2}(T): g \geq 0\ \text{a.e.}\ \}
\end{eqnarray*}
which is the positive cone of $\Lambda_{\sigma_1, \sigma_2}(T)$.\\

 Let us admit the following assumptions on coagulation, collision and breakup kernels which are required further to show the existence of solutions to  (\ref{ccfe})--(\ref{in1}).\\

\textbf{Assumptions}: $(A0)$ Let $K(x, y)$ and $C(x, y)$ be non-negative and continuous for all $(x, y) \in \mathds{R}_{>0} \times \mathds{R}_{>0} $ and $B(x|y;z)$ also non-negative and continuous for all $(x, y, z) \in \mathds{R}_{>0} \times \mathds{R}_{>0} \times \mathds{R}_{>0}$.\\
\\
$(A1)$ $K(x,y)\leq k_1 \frac{ (1+x+y)^{\mu}}{(xy)^{\sigma}}$ for all $(x, y)\in \mathbb{R}_{>0} \times \mathbb{R}_{>0}$, $0 \leq \mu-\sigma  \leq  1 $, $\sigma \in [0, 1)$ and some constant $k_1>0$,\\
\\
$(A2)$ $C(x, y)\leq k_2(1+x)^{\alpha}(1+y)^{\alpha}$ for all $(x, y)\in \mathbb{R}_{>0} \times \mathbb{R}_{>0}$,\ $0 \leq \alpha \leq 1$ and for some constant $k_2 \geq 0$. In addition,  $K(x, y)$ and $C(x, y)$ are related in the unit square by the following condition:
\begin{eqnarray*}
K(x, y) \geq 2(N-1)C(x, y),\ \ \ \forall (x, y) \in (0, 1) \times (0, 1),
\end{eqnarray*}
\\
$(A3)$ $B(x|y;z)\leq \tilde{B}\frac{1}{x}$, for $ x \in (0, y)$ and $\tilde{B}>0$,\\
\\
$(A4)$ for $p > 1$ and there are some positive constants $\omega_p <1$ such that
\begin{align*}
\int_{0}^{y}x^p B(x|y;z)dx \leq \omega_p y^p,
\end{align*}
\\
$(A5)$ for $ \omega \in [0, 1)$ and there are some positive constants $\eta(\omega) >1$ such that
\begin{align*}
\int_{0}^{y}x^{-\omega} B(x|y;z)dx \leq \eta(\omega) y^{-\omega}.
\end{align*}
Note that, in $(A5)$ for $\omega=0$, we have $\eta(\omega) \geq \sup_{y \in \mathds{R}_{>0}} N(y)=N$.

Let us consider an example of breakup kernel as $B(x|y;z)=\frac{\nu+2}{y} \bigg( \frac{x}{y}\bigg)^{\nu}$, where $ \nu \in (-1, 0]$ which clearly satisfies $(A3)$-- $(A5)$. Now, let us end up this section by stating the following theorems on existence, mass-conservation and uniqueness of solutions to (\ref{ccfe})--(\ref{in1}).

\begin{thm}\label{Existence Thm}
Suppose $(A0)$--$(A5)$ hold.  Let the initial value $g(x, 0) \in \Lambda^+_{\sigma_1, \sigma_2}(0)$. Then (\ref{ccfe})--(\ref{in1}) have a solution in $\Lambda^+_{\sigma_1, \sigma_2}(T)$.
\end{thm}

The mass-conservation property of solutions to (\ref{ccfe})--(\ref{in1}) is shown in the following theorem.

\begin{thm}\label{MassThm}
Suppose $g \in \Lambda_{\sigma_1, \sigma_2}^{+}(T)$ is a solution to (\ref{ccfe})--(\ref{in1}) with $g(x, 0) \in \Lambda^+_{\sigma_1, \sigma_2}(0)$, where $\sigma_1 \geq 2$. Assume $(A0)$--$(A5)$ hold good. Then, $g$ is a mass conserving solution.
\end{thm}

In order to investigate the uniqueness of solutions, we need further the following additional restrictions on coagulation and collision kernels.\\

$(A1^{'})$ $K(x,y)\leq k_1 \frac{ (1+x+y)^{\mu}}{(xy)^{\sigma}}$ for all $(x, y)\in \mathbb{R}_{>0} \times \mathbb{R}_{>0}$, $0 \leq \mu-\sigma  \leq  1 $, $\sigma \in [0, 1)$ and some constant $k_1>0$ and $\sigma+\theta  \leq \sigma_2$ with $\sigma \leq \theta$, where $\theta \in [0, 1)$,\\

$(A2^{'})$ $C(x, y)\leq k_2(1+x)^{\alpha}(1+y)^{\alpha}$ for all $(x, y)\in \mathbb{R}_{>0} \times \mathbb{R}_{>0}$,\ $0 \leq \alpha \leq 1 $ such that  $\alpha+1 \leq \sigma_1$ and for some constant $k_2 \geq 0$.\\

Now, we state the following uniqueness theorem
\begin{thm}\label{UniqueThm}
Suppose $g \in \Lambda_{\sigma_1, \sigma_2}^{+}(T)$ is a solution to (\ref{ccfe})--(\ref{in1}) with $g(x, 0) \in \Lambda_{\sigma_1, \sigma_2}^{+}(0)$. Let $(A0)$, $(A1^{'})$, $(A2^{'})$, $(A2)$--$(A5)$ hold. Then, $g$ is a unique solution to (\ref{ccfe})--(\ref{in1}).
\end{thm}

The following interesting examples of coagulation kernels are considered which satisfy $(A0)$--$(A1)$ as well as $(A1^{'})$.\\

$(a)$ Smoluchowski kernel for Brownian motion (continuum regime) \cite{Aldous:1999}\\
\begin{eqnarray*}
 K(x, y)=k(x^{1/3}+y^{1/3})(x^{-1/3}+y^{-1/3}).
\end{eqnarray*}
 $(b)$ Brownian motion ( free molecular regime) \cite{Aldous:1999}\\
\begin{eqnarray*}
 K(x, y)=k(x^{1/3}+y^{1/3})^2 \sqrt{1/{x}+1/{y}}.
\end{eqnarray*}
 $(c)$ Granulation kernel \cite{Camejo:2013}
\begin{eqnarray*}
K(x, y)=k\frac{(x+y)^a}{(x y )^b},\ \ \ \text{for}\ a-b\in [0, 1]\ \text{and}\ b \in [0,1).
\end{eqnarray*}

\section{Existence and Mass Conservation}

In this section, first, we construct a sequence of continuous kernels $(K_n)_{n=1}^{\infty} $ and $(C_n)_{n=1}^{\infty} $ with compact support for each $n \geq 1$, such that
\begin{equation}\label{coakernel1}
K_n(x,y)= K(x, y)\,  \text{if} \ \frac{1}{n} \leq x, y < n,
\end{equation}

\begin{equation}\label{colkernel1}
C_n(x,y)= C(x, y)\,  \text{if}\ \ \frac{1}{n} \leq x, y < n,
\end{equation}

\begin{equation}\label{coakernel2}
K_n(x,y)\leq  K(x, y)\,  \text{if}\  \  0< x, y < \infty,
\end{equation}

and
\begin{equation}\label{colkernel2}
C_n(x,y)\leq  C(x, y)\,  \text{if}\  \  0< x, y < \infty.
\end{equation}

Thus, the truncated form to the continuous coagulation and nonlinear multiple fragmentation equations (\ref{ccfe})--(\ref{in1}) can be written as
\begin{align}\label{trunc1}
\frac{\partial g_n(x,t)}{\partial t}  = & \frac{1}{2} \int_{0}^{x} K_n(x-y,y)g_n(x-y,t)g_n(y,t)dy - \int_{0}^{\infty} K_n(x,y)g_n(x,t)g_n(y,t)dy\nonumber\\
 & +\int_{x}^{\infty}\int_{0}^{\infty}B(x|y;z)C_n(y,z)g_n(z,t)g_n(y,t)dzdy\nonumber\\
 & - \int_{0}^{\infty} C_n(x,y)g_n(x,t)g_n(y,t)dy,
\end{align}
with initial data
\begin{eqnarray}\label{init}
g(x, 0)=g_n(x, 0).
\end{eqnarray}

 From the work of Stewart \cite{Stewart:1989}, Camejo \cite{Camejo:2013} and Camejo et. al. \cite{Camejo:2014I}, we can construct a sequence of nonnegative continuous unique solutions $(g_n)_{n=1}^{\infty}$ to (\ref{trunc1})--(\ref{init}) for sequence of functions of kernels $(K_n)_{n=1}^{\infty}$ and $(C_n)_{n=1}^{\infty}$ with compact support from (\ref{coakernel1}) and (\ref{colkernel1}). These solutions $(g_n(x,t))_{n=1}^{\infty}$ belong to space $\Lambda ^+_{\sigma_1, \sigma_2}(T)$. Moreover, $g_n(x,t)$ satisfies the following
 \begin{eqnarray}\label{truncmass}
\int_0^{\infty} xg_n(x,t)dx \leq \int_0^{\infty} x g_ n(x, 0)dx.
\end{eqnarray}

Next, the following lemma is required for proving the subsequent results.
\begin{lem}\label{uniboundlemma1}(Uniform boundedness of truncated moments)
Let $(A0)$--$(A5)$ hold good and $g_n(x, 0) \in \Lambda^{+}_{\sigma_1, \sigma_2}(0)$. Suppose $g_n(x, t)$ be the solution to (\ref{trunc1})--(\ref{init}). Then, for each $r \in \mathbb{R}$ and $ r \in (-1, \sigma_1] $,
\begin{eqnarray}\label{momentfinite}
\mathcal{M}_{r, n}(t):= \int_0^{\infty} x^r g_n(x, t)dx \leq \mathcal{P}_r(T),
\end{eqnarray}
for $r=1$, $\mathcal{P}_1(T)=\mathcal{P}_1$.
\end{lem}

\begin{proof}

Let $r=1$. Then by the direct integration of (\ref{trunc1}) with respect to $x$ from $0$ to $\infty$ after multiplying with the weight $x$, we get
\begin{align}\label{unibound1}
\frac{d}{dt} \mathcal{M}_{1,n}(t)= \int_0^{\infty}x &\bigg[\frac{1}{2}\int_{0}^{x} K_n(x-y,y)g_n(x-y,t)g_n(y,t)dy- \int_{0}^{\infty} K_n(x,y)g_n(x,t)g_n(y,t)dy\nonumber\\
 & +\int_{x}^{\infty}\int_{0}^{\infty} B(x|y;z)C_n(y, z)g_n(y, t)g_n(z, t)dzdy\nonumber\\
& - \int_{0}^{\infty} C_n(x, y)g_n(x, t)g_n(y, t)dy \bigg]dx.
\end{align}

Due to the compact support on $K_n$ and $C_n$, all integrals obtained on the right-hand side of (\ref{unibound1}) are finite. Then, by using Fubini's theorem, it can be seen that the first and second integrals on the right-hand side of (\ref{unibound1}) are equal. Similarly, using Fubini's theorem and (\ref{mass1}), the third and fourth integrals on the right-hand side of (\ref{unibound1}) are exactly same. Hence, they cancel out. Therefore, we find $\mathcal{M}_{1,n}(t) :=\mathcal{P}_1$ (say) is a constant.\\

 Next, we deduce the boundedness of the zeroth moment by integrating (\ref{trunc1}) with respect to $x$ from $0$ to $\infty$ as
\begin{align}\label{unibound2}
\frac{d}{dt} \mathcal{M}_{0,n}(t)= \int_0^{\infty}& \bigg[\frac{1}{2} \int_{0}^{x} K_n(x-y, y)g_n(x-y, t)g_n(y, t)dy- \int_{0}^{\infty} K_n(x, y)g_n(x, t)g_n(y, t)dy\nonumber\\
 & +\int_{x}^{\infty}\int_{0}^{\infty} B(x|y;z)C_n(y, z)g_n(y, t)g_n(z, t)dzdy\nonumber\\
 &- \int_{0}^{\infty} C_n(x,y)g_n(x, t)g_n(y, t)dy\bigg]dx.
\end{align}
Applying $(A2)$, (\ref{N1}) and Fubini's theorem to the first and third integrals on the right-hand side of (\ref{unibound2}), we estimate
\begin{align*}
\frac{d}{dt} \mathcal{M}_{0,n}(t)
 \leq &-\frac{1}{2}\int_0^{1} \int_{0}^{1} [K_n(x, y)-2(N-1)C_n(x, y)]g_n(x, t)g_n(y, t)dydx\nonumber\\
 &+ k_2(N-1)\int_0^{1} \int_{1}^{\infty} (1+x)^{\alpha} (1+y)^{\alpha} g_n(x, t)g_n(y, t)dydx\nonumber\\
 & +k_2(N-1)\int_1^{\infty} \int_{0}^{1}(1+x)^{\alpha} (1+y)^{\alpha} g_n(x, t)g_n(y, t)dydx\nonumber\\
  & +k_2(N-1)\int_1^{\infty} \int_{1}^{\infty} (1+x)^{\alpha} (1+y)^{\alpha} g_n(x, t)g_n(y, t)dydx.
  \end{align*}
 Using (\ref{truncmass}) to the above inequality, we have
  \begin{align*}
\frac{d}{dt} \mathcal{M}_{0,n}(t)   \leq & 8k_2(N-1)\int_0^{1} \int_{1}^{\infty}yg_n(x,t)g_n(y,t)dydx\nonumber\\
    &+4k_2(N-1)\int_1^{\infty} \int_{1}^{\infty} xy g_n(x,t)g_n(y,t)dydx\nonumber\\
   \leq & 8k_2(N-1)\mathcal{P}_1 \mathcal{M}_{0,n}(t) +4k_2(N-1)\mathcal{P}_1^2.
\end{align*}
Therefore, an application of Gronwall's inequality gives
\begin{align*}
\mathcal{M}_{0,n}(t) \leq   [ \mathcal{M}_{0,n}(0)+4k_2(N-1){ \mathcal{P}_1}^2T] e^{8k_2(N-1)\mathcal{P}_1 T} =: \mathcal{P}_0(T).
\end{align*}

Next, for $r \geq 2$, we use induction on $n$, let us first estimate the truncated moment for $r=2$ as
\begin{align}\label{uniboundr2}
\frac{d}{dt} \mathcal{M}_{2,n}(t)=  \int_0^{\infty}x^2 &\bigg[\frac{1}{2}\int_{0}^{x} K_n(x-y,y)g_n(x-y,t)g_n(y,t)dy- \int_{0}^{\infty} K_n(x,y)g_n(x,t)g_n(y,t)dy\nonumber\\
 & +\int_{x}^{\infty}\int_{0}^{\infty}B(x|y;z)C_n(y,z)g_n(z,t)g_n(y,t)dzdy\nonumber\\
& - \int_{0}^{\infty} C_n(x,y)g_n(x,t)g_n(y,t)dy \bigg]dx.
\end{align}
Changing the order of integration to the first and third integrals on right-hand side to (\ref{uniboundr2}) and then applying $(A4)$, we obtain
\begin{align}\label{unibound3}
\frac{d}{dt} \mathcal{M}_{2,n}(t) =& \int_0^{\infty} \int_{0}^{\infty}xy  K_n(x,y)g_n(x,t)g_n(y,t)dydx\nonumber\\
& -(1-\omega_2) \int_{0}^{\infty}\int_{0}^{\infty}x^2 C_n(x,y)g_n(x,t)g_n(y,t)dydx.
\end{align}
The second integral on the right-hand side of (\ref{unibound3}) is non-negative. Therefore, using $(A1)$, we have
\begin{align*}
\frac{d}{dt} \mathcal{M}_{2,n}(t) 
\leq &  k_1 \int_0^{1} \int_{0}^{1}(xy)^{1-\sigma} (1+x+y)^{\mu}g_n(x,t)g_n(y,t)dydx\nonumber\\
& + 2k_1 \int_0^{1} \int_{1}^{\infty}(xy)^{1-\sigma} (1+x+y)^{\mu}g_n(x,t)g_n(y,t)dydx\nonumber\\
&  +k_1 \int_1^{\infty} \int_{1}^{\infty}(xy)^{1-\sigma} (1+x+y)^{\mu}g_n(x,t)g_n(y,t)dydx.
\end{align*}
Next, using $(x+y)^{\mu} \leq k(\mu)(x^{\mu}+ y^{\mu})$, for $x>0$, $y>0$ and $\mu \geq 0$ to the above inequality, we get
\begin{align}\label{unibound5}
\frac{d}{dt} \mathcal{M}_{2,n}(t)  \leq &  3^{\mu} k_1 \int_0^{1} \int_{0}^{1}  g_n(x,t)g_n(y,t)dydx\nonumber\\
& +  3^{\mu}  2k_1 \int_0^{1} \int_{1}^{\infty} y^{1+\mu-\sigma}  g_n(x,t)g_n(y,t)dydx\nonumber\\
&  +k_1 \int_1^{\infty} \int_{1}^{\infty}(xy)^{1-\sigma} (1+x+y)^{\mu}g_n(x,t)g_n(y,t)dydx\nonumber\\
\leq &  3^{\mu} k_1 \mathcal{P}_0(T)^2 +  3^{\mu}  2k_1 \mathcal{P}_0(T) \mathcal{M}_{2,n}(t) +k_1 k(\mu) [\mathcal{P}_1^2+2\mathcal{P}_1 \mathcal{M}_{2,n}(t)   ].
\end{align}
Then, applying Gronwall's inequality to (\ref{unibound5}), we obtain
\begin{align*}
\mathcal{M}_{2,n}(t) \leq \mathcal{P}_2(T),
\end{align*}
where $\mathcal{P}_2(T) := \mathcal{M}_{2,n}(0) e^{(3^{\mu}\mathcal{P}_0(T)+k(\mu)\mathcal{P}_1)2k_1T }+ \frac{1}{2} \frac{ 3^{\mu}  \mathcal{P}_0(T)^2+k(\mu)\mathcal{P}_1^2 }{ 3^{\mu} \mathcal{P}_0(T)+k(\mu)\mathcal{P}_1 }( e^{(3^{\mu}\mathcal{P}_0(T)+k(\mu)\mathcal{P}_1 )2k_1T} -1 )$.\\

Let \begin{eqnarray}\label{mk}
\mathcal{M}_{r,n}(t)\leq \mathcal{P}_r(T)\ \ \text{holds}\ \text{for}\ \ r=k\geq 2.
\end{eqnarray}
Then it is sufficient to show that $\mathcal{M}_{r,n}(t)\leq \mathcal{P}_r(T)$, for $r=k+1$. For this purpose, we multiplying (\ref{trunc1}) by $x^{k+1}$ and taking integration with respect to $x$ between $0$ and $\infty$ to have
\begin{align}\label{unibound6}
\frac{d}{dt}\mathcal{M}_{k+1, n}(t)=&\frac{1}{2} \int_0^{\infty} \int_0^{\infty} [(x+y)^{k+1}-x^{k+1}-y^{k+1}]K_n(x,y)g_n(x,t)g_n(y,t)dydx\nonumber\\
&-(1-\omega_{k+1}) \int_0^{\infty} \int_0^{\infty} x^{k+1}C_n(x,y)g_n(x,t)g_n(y,t)dydx.
\end{align}
Using the non-negativity of the second integral in (\ref{unibound6}) and then applying binomial theorem, we estimate
\begin{align}\label{unibound7}
\frac{d}{dt}\mathcal{M}_{k+1, n}(t)\leq &\frac{1}{2} \int_0^{\infty} \int_0^{\infty} \bigg[ \binom{k+1} {1}x^k y +\underbrace{ \binom{k+1} {2}x^{k-1} y^2+\cdots \binom{k+1}{k-1}x^2y^{k-1}}+\binom{k+1} {k}x y^k \bigg]\nonumber\\
& \times  K_n(x,y)g_n(x,t)g_n(y,t)dydx.
\end{align}

The first and last integrals on  the right-hand side of (\ref{unibound7}) are equal. Therefore, we can merge them. By using $(A1)$ and (\ref{mk}) rest of the integrals on the right-hand side of (\ref{unibound7}) can easily be shown finite individually.
\begin{align}\label{unibound71}
\frac{d}{dt}\mathcal{M}_{k+1, n}(t)\leq & \Omega_1(t) +\Omega(T),
\end{align}
where $\Omega_1(t):=\int_0^{\infty} \int_0^{\infty} \binom{k+1} {1}x^k y K_n(x,y)g_n(x,t)g_n(y,t)dydx $ and $\Omega(T)(\text{finite})=$ rest of integrals except first and last integrals of (\ref{unibound7}).
Next, by using $(A1)$, evaluate $\Omega_1(t)$ as
\begin{align}\label{unibound8}
 \Omega_1(t)
 =&\binom{k+1} {1} k_1 \int_0^1 \int_0^1 x^{k-\sigma} y^{1-\sigma}  (1+x+y)^{\mu}   g_n(x,t)g_n(y,t)dydx\nonumber\\
&+2\binom{k+1} {1} k_1 \int_0^1 \int_1^{\infty} x^{k-\sigma} y^{1-\sigma}  (1+x+y)^{\mu}   g_n(x,t)g_n(y,t)dydx\nonumber\\
&+\binom{k+1} {1} k_1 \int_1^{\infty} \int_1^{\infty} x^{k-\sigma} y^{1-\sigma}  (1+x+y)^{\mu}   g_n(x,t)g_n(y,t)dydx\nonumber\\
 \leq &\binom{k+1} {1} 3k_1 \mathcal{P}_0(T)^2 +2\binom{k+1} {1} k_1 3^{\mu} \mathcal{P}_0(T) \int_1^{\infty}   y^{1+\mu-\sigma}  g_n(y,t)dy\nonumber\\
&+\binom{k+1} {1} k_1 k(\mu) \int_1^{\infty} \int_1^{\infty} [x^{k-\sigma} y^{1-\sigma}+ x^{k+\mu-\sigma} y^{1-\sigma}+x^{k-\sigma} y^{1+\mu-\sigma}]  g_n(x,t)g_n(y,t)dydx\nonumber\\
\leq &\binom{k+1} {1} 3k_1 \mathcal{P}_0(T)^2 +2\binom{k+1} {1} k_1 3^{\mu} \mathcal{P}_0(T) \mathcal{P}_2(T)\nonumber\\
&+\binom{k+1} {1} k_1 k(\mu) [ \mathcal{P}_k(T) \mathcal{P}_1+ \mathcal{P}_1 \mathcal{M}_{k+1, n}(t)+ \mathcal{P}_k(T)\mathcal{P}_2(T)].
\end{align}
Inserting values of (\ref{unibound8}) into (\ref{unibound7}), we obtain
\begin{align}\label{unibound9}
\frac{d}{dt}\mathcal{M}_{k+1, n}(t) \leq &\binom{k+1} {1} 3k_1 \mathcal{P}_0(T)^2 +2\binom{k+1} {1} k_1 3^{\mu} \mathcal{P}_0(T) \mathcal{P}_2(T)\nonumber\\
&+\binom{k+1} {1} k_1 k(\mu) [ \mathcal{P}_k(T) \mathcal{P}_1+ \mathcal{P}_1 \mathcal{M}_{k+1, n}(t)+ \mathcal{P}_k(T)\mathcal{P}_2(T)] +\Omega(T)\nonumber\\
= &\binom{k+1} {1} k_1 k(\mu)  \mathcal{P}_1 \mathcal{M}_{k+1, n}(t)  +\Omega_2(T),
\end{align}
where $\Omega_2(T):= \Omega(T)+ \binom{k+1} {1} 3k_1 \mathcal{P}_0(T)^2 +2\binom{k+1} {1} k_1 3^{\mu} \mathcal{P}_0(T) \mathcal{P}_2(T)+\binom{k+1} {1} k_1 k(\mu)  \mathcal{P}_k(T) \mathcal{P}_1+\binom{k+1} {1} k_1 k(\mu)\mathcal{P}_k(T)\mathcal{P}_2(T)$.
After applying Gronwall's inequality in (\ref{unibound9}), we obtain
\begin{align}
\mathcal{M}_{k+1, n}(t) \leq \mathcal{P}_{k+1}(T),
\end{align}
where $\mathcal{P}_{k+1}(T):= e^{bT}(\frac{b}{\Omega_2(T)}+ \mathcal{M}_{k+1, n}(0))- \frac{b}{\Omega_2(T)}$, and $b:= \binom{k+1} {1} k_1 k(\mu)  \mathcal{P}_1$.\\

Next, we check the uniform boundedness of $\mathcal{M}_{-\omega, n}(t)$ for $\omega \in (\sigma, 1)$. For this, we multiply (\ref{trunc1}) by $x^{-\omega}$ and integrating with respect to $x$ from $0$ to $\infty$ and using $(A6)$ to achieve
\begin{align}\label{unibound11}
\frac{d}{dt}\mathcal{M}_{-\omega, n}(t) = &\frac{1}{2} \int_0^{\infty} \int_0^{\infty} \bigg[ (x+y)^{-\omega}-x^{-\omega}-y^{-\omega} \bigg] K_n(x, y)g_n(x,t)g_n(y, t)dy dx\nonumber\\
&+(\eta(\omega)-1) \int_0^{\infty} \int_0^{\infty} x^{-\omega} C_n(x, y)g_n(x,t)g_n(y, t)dy dx.
\end{align}
We know that $(x+y)^{-\omega} \leq x^{-\omega}$ which implies the non-positivity of the first term on the right-hand side to (\ref{unibound11}).
Therefore, using in (\ref{unibound11}) $(A2)$, we estimate
\begin{align}\label{unibound12}
\frac{d}{dt}\mathcal{M}_{-\omega, n}(t)\leq & (\eta(\omega)-1) \int_0^{\infty} \int_0^{\infty} x^{-\omega} C_n(x, y)g_n(x,t)g_n(y, t)dy dx\nonumber\\
\leq & (\eta(\omega)-1)k_2 [\mathcal{P}_0(T)+\mathcal{P}_1]  \bigg[ \int_0^1 x^{-\omega} (1+x^{\alpha})  g_n(x, t)dx+\int_1^{\infty} x^{-\omega} (1+x^{\alpha})  g_n(x, t)dx  \bigg] \nonumber\\
\leq & 2(\eta(\omega)-1)k_2 [\mathcal{P}_0(T)+\mathcal{P}_1]  \bigg[ \mathcal{M}_{-\omega, n}(t) + \mathcal{P}_1  \bigg].
\end{align}
Applying Gronwall's inequality in (\ref{unibound12}), we obtain
\begin{align*}
\mathcal{M}_{-\omega, n}(t)\leq \mathcal{P}_{-\omega}(T),
\end{align*}
where $\mathcal{P}_{-\omega}(T):= e^{(  2(\eta(\omega)-1)k_2 [\mathcal{P}_0(T)+\mathcal{P}_1]T  )}( \frac{1}{\mathcal{P}_1  } +\mathcal{M}_{-\omega, n}(0) )- \frac{1}{\mathcal{P}_1  }$.
This completes the prove of lemma.

\end{proof}

Next, we require the following lemma.

\begin{lem}\label{weakcompactlemma}
Assume $(A0)$--$(A5)$ hold. Let $g_n(x, t)$ be a solution to (\ref{trunc1})--(\ref{init}) with $g_n(x, 0) \in \Lambda^{+}_{\sigma_1, \sigma_2}(0)$. Then prove that the sequence $\{g_n\}_{n\in \mathds{N}}$ is relatively compact in the uniform convergence topology of continuous functions on each rectangle $\Xi (\lambda_1, \lambda_2; T)$.
\end{lem}
\begin{proof}
For proving Lemma \ref{weakcompactlemma}, we proceed to show the following results:\\

$(\alpha_1)$ the uniform boundedness of the sequence $\{g_n\}_{n\in \mathds{N}}$,\\
\\
$(\alpha_2)$ then the equicontinuity of the sequence $\{g_n\}_{n\in \mathds{N}}$ with respect to the time variable $t$,\\
\\
$(\alpha_3)$ and finally the equicontinuity of the sequence $\{g_n\}_{n\in \mathds{N}}$ with respect to the space variable $x$.\\

\emph{Proof of $(\alpha_1)$}: Here, we show $\{g_n\}_{n \in \mathds{N}}$ is uniformly bounded on $\Xi (\lambda_1, \lambda_2; T)$, where $T>0$ is fixed.

Since the second and fourth terms on right-hand side of (\ref{trunc1}) are non-negative, we have
\begin{align}\label{uniform2}
\frac {\partial g_n(x,t)}{\partial t} \leq &\frac{1}{2}\int_{0}^{x} K_n(x-y,y)g_n(x-y,t)g_n(y,t)dy\nonumber\\
&+\int_{x}^{\infty}\int_{0}^{\infty}B(x|y;z)C_n(y,z)g_n(y,t)g_n(z,t)dzdy.
\end{align}
Using $(A1)$--$(A3)$ and Lemma \ref{uniboundlemma1} in (\ref{uniform2}), we estimate
\begin{align}\label{uniform3}
\frac {\partial g_n(x,t)}{\partial t} \leq &\frac{1}{2} k_1 \int_{0}^{x} \frac{(1+x)^{\mu}}{(x-y)^{\sigma}y^{\sigma}}  g_n(x-y,t)g_n(y,t)dy\nonumber\\
&+k_2 \int_{x}^{\infty}\int_{0}^{\infty}   B(x|y;z) (1+y)^{\alpha}(1+z)^{\alpha} g_n(y,t)g_n(z,t)dzdy\nonumber\\
\leq &\frac{1}{2} k_1 (1+\lambda_2)^{\mu}  (c_n * c_n)(x,t)+k_2 \frac{\tilde{B}}{\lambda_1} [\mathcal{P}_0(T)+\mathcal{P}_1 ]^2,
\end{align}
where $c_n=\frac{g_n(x, t)}{x^{\sigma}} $ and $(c_n * c_n)(x,t)$ is the convolution between $c_n$ with itself. Multiplying $x^{-\sigma}$ and then integrating to (\ref{uniform3}) from $0$ to $t$ with respect to time variable, we obtain
\begin{eqnarray}\label{uniform4}
c_n(x,t) \leq  \tilde{c}(0) + \lambda_1^{-\sigma} \int_0^t \bigg[\frac{1}{2} k_1 (1+\lambda_2)^{\mu} (c_n * c_n)(x, s) +k_2 \frac{\tilde{B}}{\lambda_1} [\mathcal{P}_0(T)+\mathcal{P}_1 ]^2\bigg] ds,
\end{eqnarray}
where $\tilde{c}(0)=\sup_{\lambda_1 \leq x \leq \lambda_2}c_n(x, 0)$.\\

 Let us now define a function to control the right-hand side of the integral inequality (\ref{uniform4}) as
\begin{eqnarray}\label{uniform5}
E(x,t) :=  E(0) +  \int_0^t \bigg[\frac{1}{2} k_1 (1+\lambda_2)^{\mu} \lambda_1^{-\sigma} (E * E)(x, s) + E(x, s) \bigg] ds,
\end{eqnarray}

where $E(0)= \max\{ \tilde{c}(0), k_2 \lambda_1^{-\sigma-1} \tilde{B} [\mathcal{P}_0(T)+\mathcal{P}_1 ]^2 \}$ is a positive constant. Then, applying Laplace transform and then its inverse transformation to (\ref{uniform5}) with respect to $x$, we have
\begin{eqnarray}\label{uniform6}
E(x,t) =  E(0) \exp \bigg( \frac{1}{2} k_1E(0)x (1+\lambda_2)^{\mu} \lambda_1^{-\sigma} (e^t-1)+t \bigg),\ \ \lambda_1 \leq x \leq \lambda_2, 0\leq t \leq T.
\end{eqnarray}
In order to complete the proof of the first part of Lemma \ref{weakcompactlemma}, it is required to show that the following inequality hold
\begin{eqnarray}\label{uniform7}
c_n(x, t) \leq E(x, t),\ \ \ \lambda_1 \leq x \leq \lambda_1,\ 0 \leq t \leq T,\ n \in \mathds{N}.
\end{eqnarray}
(\ref{uniform7}) is shown by a contradiction. Next, define the following auxiliary function as
\begin{eqnarray}\label{uniform8}
E_{\epsilon}(x,t):= E(0)+\epsilon + \int_0^t \bigg[\frac{1}{2} k_1 (1+\lambda_2)^{\mu} \lambda_1^{-\sigma} (E_{\epsilon} * E_{\epsilon})(x, s) + E_{\epsilon}(x, s) \bigg] ds.
\end{eqnarray}
From (\ref{uniform4}) and (\ref{uniform5}), it is clear that $c_n(x, 0) \leq  E_{\epsilon}(x, 0)$, for $\lambda_1 \leq x \leq \lambda_2$. Let us assume that, for $n \geq 1$, there exists a set $A$ such that $c_n(x, t) = E_{\epsilon}(x, t)$ for $(x, t)\in \Xi (\lambda_1, \lambda_2; T)$. Let us choose $(x_1, t_1)\in A$ such that no points of $D:=[\lambda_1, x_1)\times [0, t_1) $ in $A$. Since $c_n$ and $E_{\epsilon}$ are continuous function with respect to time variable, we obtain
\begin{align}\label{uniform9}
c_n(x_1, t_1)=E_{\epsilon}(x_1, t_1)= & E_{\epsilon}(0)+\epsilon + \int_0^{t_1} \bigg[\frac{1}{2} k_1 (1+\lambda_2)^{\mu} \lambda_1^{-\sigma} (E_{\epsilon} * E_{\epsilon})(x, s) + E_{\epsilon}(x, s) \bigg] ds\nonumber\\
\geq & c_n(x_1, 0)+\epsilon + \int_0^{t_1} \bigg[\frac{1}{2} k_1 (1+\lambda_2)^{\mu} \lambda_1^{-\sigma} (c_n * c_n)(x, s) + c_n(x, s) \bigg] ds\nonumber\\
> & c_n(x_1, t_1),
\end{align}
which is a contradiction. This concludes that $A$ is empty set.
This gives
\begin{eqnarray}\label{uniform10}
c_n(x, t) \leq E(x, t),\ \ \ \lambda_1 \leq x \leq \lambda_1,\ 0 \leq t \leq T,\ n \in \mathds{N}.
\end{eqnarray}
Thus, we have
\begin{eqnarray}\label{S(T)}
g_n(x, t) \leq S(T),
\end{eqnarray}
where $S(T):=E(0) \exp \bigg( \frac{1}{2}E(0) k_1 (1+\lambda_2)^{\mu} \lambda_1^{1-\sigma} (e^T-1)+T \bigg) $.
Hence, the sequence $\{g_n\}_{n \in \mathds{N}}$ is uniformly bounded on $\Xi(\lambda_1, \lambda_2; T).$\\


\emph{Proof of $(\alpha_2)$}: We next establish the time equicontinuity of $\{g_n\}_{n=1}^{\infty}$ in the rectangle $\Xi(\lambda_1, \lambda_2; T)$. Let us assume $0\leq \tau_1 \leq \tau_2 \leq T$, for each $n\geq 1$.\\
By the definition of equicontinuity, for an arbitrary $\epsilon >0$, there exists a small positive number $\delta (\epsilon)$ for which
\begin{align}\label{eqcont1}
|g_n(x, \tau_2)-g_n(x, \tau_1))|< \epsilon,   \mbox{~ whenever~} |\tau_2-\tau_1|<\delta(\epsilon).
\end{align}

From (\ref{trunc1}), we have
\begin{align}\label{eqcont2}
|g_n(x, \tau_2-g_n(x, \tau_1))| \leq \int_{\tau_1}^{\tau_2} & \bigg[\frac{1}{2}\int_{0}^{x} K_n(x-y, y)g_n(x-y, t)g_n(y, t)dy\nonumber\\
&+ \int_{0}^{\infty} K_n(x, y)g_n(x, t)g_n(y, t)dy\nonumber\\
&+ \int_{x}^{\infty}\int_{0}^{\infty}B(x|y;z)C_n(y,z)g_n(y, t)g_n(z,t)dzdy\nonumber\\
&+ \int_{0}^{\infty} C_n(x, y)g_n(x, t)g_n(y, t)dy\bigg]dt.
\end{align}

Let us simplify the first integral on the right-hand side in (\ref{eqcont2}) by using first part of Lemma \ref{weakcompactlemma} and $(A1)$, as
\begin{align}\label{eqcont3}
\frac{1}{2}\int_{\tau_1}^{\tau_2} \int_{0}^{x} K_n(x-y,y)g_n(x-y, t)g_n(y, t)dydt \leq &  \frac{1}{2}k_1 \int_{\tau_1}^{\tau_2}\int_{0}^{x} \frac{(1+x)^{\mu} }{(x-y)^{\sigma}y^{\sigma}} g_n(x, t)g_n(y, t)dydt\nonumber\\
\leq &  \frac{1}{2}k_1S(T)^2 (1+\lambda_2)^{\mu} \int_{\tau_1}^{\tau_2}\int_{0}^{x} \frac{1 }{(x-y)^{\sigma}y^{\sigma}}  dyds\nonumber\\
\leq & \frac{k_1}{2}S(T)^2 (1+\lambda_2)^{\mu} \Theta (\tau_2-\tau_1),
\end{align}
where $\Theta:= \frac{ (\Gamma(1-\sigma))^2 \lambda_2^2 \lambda_1^{-2\sigma} }{ \Gamma (2-2\sigma)}$ and $\Gamma$ is the gamma function.
By using Lemma \ref{uniboundlemma1} and (\ref{S(T)}), the second integral on the right-hand side of (\ref{eqcont2}) can be simplified as
\begin{align}\label{eqcont4}
\int_{\tau_1}^{\tau_2} \int_{0}^{\infty} K_n(x, y) & g_n(x, t)g_n(y, t) dydt \leq  k_1 \int_{\tau_1}^{\tau_2} \int_{0}^{\infty} \frac{(1+x+y)^{\mu} }{(xy)^{\sigma}} g_n(x, t)g_n(y, t)dydt \nonumber\\
\leq & k_1 \frac{(1+\lambda_2)^{\mu}}{\lambda_1^{\sigma}} S(T)2^{\mu} \int_{\tau_1}^{\tau_2} \bigg[ \int_{0}^{1} y^{-\sigma} g_n(y, t)dy+\int_{1}^{\infty} y g_n(y, t)dy \bigg]  dt\nonumber\\
\leq & k_1 \frac{(1+\lambda_2)^{\mu}}{\lambda_1^{\sigma}} S(T) (\mathcal{P}_{-\sigma}(T)+\mathcal{P}_1) (\tau_2 -\tau_1).
\end{align}

Using $(A2)$, $(A3)$, Lemma \ref{uniboundlemma1} and (\ref{S(T)}), the third integral on the right-hand side to (\ref{eqcont2}) can be estimated as
\begin{align}\label{eqcont5}
\int_{\tau_1}^{\tau_2}\int_x^{\infty}  \int_0^{\infty} &B(x|y;z) C_n(y,z)g_n(y, t)g_n(z, t)dzdydt \nonumber\\
 \leq & k_2 \widetilde{B} \int_{\tau_1}^{\tau_2} \int_x^{\infty}  \int_0^{\infty} \frac{1}{x} (1+y)^{\alpha} (1+z)^{\alpha}  g_n(y, t)g_n(z, t)dzdydt \nonumber\\
  \leq & k_2 \widetilde{B} [\mathcal{P}_0(T)+\mathcal{P}_1] \frac{1} {\lambda_1} \int_{\tau_1}^{\tau_2} \int_x^{\infty}  (1+y) g_n(y, t) dydt \nonumber\\
  \leq & k_2 \widetilde{B} [\mathcal{P}_0(T)+\mathcal{P}_1]^2  \frac{1}{\lambda_1} ({\tau_2}-{\tau_1}).
\end{align}
The last integral on the right-hand side can be evaluated, by using $(A2)$, Lemma \ref{uniboundlemma1} and (\ref{S(T)}), as
\begin{align}\label{eqcont6}
\int_{\tau_1}^{\tau_2}\int_{0}^{\infty} C_n(x,y)g_n(x, t)g_n(y, t)dydt \leq & k_2 \int_{\tau_1}^{\tau_2} \int_{0}^{\infty} (1+x)^{\alpha}(1+y)^{\alpha} g_n(x, t)g_n(y, t)dy dt\nonumber\\
\leq & k_2 S(T)(1+\lambda_2) \int_{\tau_1}^{\tau_2} \bigg[ \int_{0}^{1} (1+y) g_n(y, t)dy\nonumber\\
&+\int_{1}^{\infty} (1+y)  g_n(y, t)dy \bigg]dt\nonumber\\
\leq & 2k_2 S(T)(1+\lambda_2) [ \mathcal{P}_0(T)+ \mathcal{P}_1 ]  ({\tau_2}-{\tau_1}).
\end{align}
Inserting  (\ref{eqcont3}), (\ref{eqcont4}), (\ref{eqcont5}) and (\ref{eqcont6}) into (\ref{eqcont2}), we conclude that
\begin{align}\label{finalequitime}
|g_n(x, \tau_2)-g_n(x, \tau_1)| \leq \Gamma_1(T) |\tau_2-\tau_1| < \Gamma_1(T) \delta,
\end{align}
where $\Gamma_1(T):=\frac{k_1}{2}S(T)^2 (1+\lambda_2)^{\mu} \Theta +  k_1 \frac{(1+\lambda_2)^{\mu}}{\lambda_1^{\sigma}} S(T) (\mathcal{P}_{-\sigma}(T)+\mathcal{P}_1)+ k_2 \tilde{B} [\mathcal{P}_0(T)+\mathcal{P}_1]^2  \frac{1}{\lambda_1}+ 2k_2 S(T)(1+\lambda_2) [ \mathcal{P}_0(T)+ \mathcal{P}_1 ] $ and for $|\tau_2-\tau_1|<\delta$.
Therefore, the sequence $\{g_n\}_{n \in \mathds{N}}$ is equicontinuous with respect to time variable $t$ on $\Xi(\lambda_1, \lambda_2; T)$.\\
\\


\emph{Proof of $(\alpha_3)$}:  Next, we establish the equicontinuity of $g_n(x,t)$ with respect to the variable $x$ in the rectangle $\Xi (\lambda_1, \lambda_2; T).$
Let $\lambda_1 \leq x_1 \leq x_2 \leq \lambda_2$ and then for each $n\geq 1$, we have
\begin{align*}
|g_n(x_2, t)-g_n(x_1, t)| \leq &|g_0(x_2, t)-g_0(x_1, t)|+ \bigg| \frac{1}{2} \int_0^t \int_{x_1}^{x_2 }K_n(x_2-y,y)g_n(x_2-y,s)g_n(y,s)dyds \bigg| \nonumber\\
&+ \frac{1}{2} \int_0^t \int_0^{x_1}|K_n(x_2-y, y)-K_n(x_1-y, y)|g_n(x_1-y, s)g_n(y, s)dyds\nonumber\\
&+\frac{1}{2} \int_0^t \int_0^{x_1} K_n(x_2-y, y)|g_n(x_2-y, s)-g_n(x_1-y, s)|g_n(y, s)dyds\nonumber\\
&+\int_0^t | g_n(x_2,s)-g_n(x_1, s)|\int_0^{\infty}K_n(x_1, y)g_n(y, s)dyds\nonumber\\
&+\int_0^t g_n(x_2, s)\int_0^{\infty} |K_n(x_2, y)-K_n(x_1, y)|g_n(y,s)dyds\nonumber\\
&+\int_0^t\int_{x_2}^{\infty}\int_0^{\infty}|B(x_2|y;z)-B(x_1|y;z)|C_n(y,z)g_n(y,s)g_n(z,s)dzdyds
\end{align*}
\begin{align}\label{eqconspace1}
&+\int_0^t\int_{x_1}^{x_2 }\int_0^{\infty}B(x_1|y;z)C_n(y,z)g_n(y,s)g_n(z,s)dzdyds\nonumber\\
&+\int_0^t | g_n(x_2, s)-g_n(x_1, s)|\int_0^{\infty}C_n(x_1, y)g_n(y,s)dyds\nonumber\\
&+ \int_0^tg_n(x_2, s)\int_0^{\infty} |C_n(x_2, y)-C_n(x_1, y)|g_n(y,s)dyds\nonumber\\
=& \sum_{i=1}^{10} \mathcal{I}_i^n(t).
\end{align}

According to the construction of kernels, sequence of kernels $\{K_n\}$ and $\{C_n\}$  are equicontinuous over the rectangle $[\lambda_1, \lambda_2] \times [Y_1, Y_2], \lambda_1, \lambda_2, Y_1, Y_2>0,$ and we have $B(x|y;z)$ is equicontinuous over the $[\lambda_1, \lambda_2] \times [Y_1, Y_2] \times [Z_1, Z_2].$ Our aim is to show that when $|x_2-x_1|$ is small enough, then the left-hand side of (\ref{eqconspace1}) is sufficiently small. Corresponding to arbitrary $\epsilon >0$, there exists a $\delta (\epsilon)>0$ with
\begin{eqnarray}\label{eqconspace2}
 \sup_{|x_2-x_1|< \delta} |g_0(x_2)-g_0(x_1)|< \epsilon,
\end{eqnarray}

\begin{eqnarray}\label{eqconspace3}
 \sup_{|x_2-x_1|< \delta} |K_n(x_2, y)-K_n(x_1, y)|< \epsilon,
 \end{eqnarray}

 \begin{eqnarray}\label{eqconspace4}
 \sup_{|x_2-x_1|< \delta} |C_n(x_2, y)-C_n(x_1, y)|< \epsilon,
 \end{eqnarray}
 and
 \begin{eqnarray}\label{eqconspace5}
 \sup_{|x_2-x_1|< \delta} |B(x_2|y;z)-B(x_1|y:z)|< \epsilon.
\end{eqnarray}

 The above inequalities, (\ref{eqconspace3})--(\ref{eqconspace5}) hold uniformly with respect to $n\geq 1$ and $Y_1 \leq y \leq Y_2$ and $Z_1 \leq z \leq Z_2$. We introduce modulus of continuity as
 \begin{eqnarray*}
 \omega _n (t) := \sup_{|x_2-x_1|< \delta} |g_n(x_2, t)-g_n(x_1, t)|,\ \ \ \lambda_1 \leq x_1,~x_2 \leq \lambda_2.
 \end{eqnarray*}
Let us first estimate $\mathcal{I}_2^n(t)$, by using (A1) and (\ref{S(T)}), as
\begin{align}\label{I2}
\mathcal{I}_2^n(t) 
\leq & \frac{1}{2}k_1 \bigg|\int_0^t  \int_{x_1}^{x_2} \frac{(1+x_2)^{\mu} } {(x_2-y)^{\sigma} y^{\sigma} } g_n(x_2-y, s)g_n(y, s)dyds\bigg| \nonumber\\
\leq & \frac{1}{2}k_1 \frac{(1+\lambda_2)^{\mu}}{\lambda_1^{\sigma}}S(T)^2T  \bigg|\int_{x_1}^{x_2} \frac{1 } {(x_2-y)^{\sigma} } dy\bigg|\nonumber\\
\leq & \frac{1}{2}k_1 \frac{(1+\lambda_2)^{\mu}}{\lambda_1^{\sigma}}S(T)^2T \bigg| \bigg[ \frac{(x_2-y)^{1-\sigma} } {1-\sigma }\bigg]_{x_1}^{x_2}\bigg| \nonumber\\
\leq & \frac{1}{2}k_1 \frac{(1+\lambda_2)^{\mu}}{\lambda_1^{\sigma}}S(T)^2T  \frac{|x_2-x_1|^{1-\sigma} } {1-\sigma }=:G_2 \delta^{1-\sigma}(\epsilon).
\end{align}
 $\mathcal{I}_3^n(t)$ can be simplified as
\begin{align}\label{I3}
\mathcal{I}_3^n(t) \leq & \frac{1}{2} \int_0^t  \int_0^{Y_1}|K_n(x_2-y,y)-K_n(x_1-y,y)|g_n(x_1-y,s)g_n(y,s)dyds\nonumber\\
&+ \frac{1}{2} \int_0^t  \int_{Y_1}^{Y_2}|K_n(x_2-y,y)-K_n(x_1-y,y)|g_n(x_1-y,s)g_n(y,s)dyds\nonumber\\
& +\frac{1}{2} \int_0^t  \int_{Y_2}^{\infty}|K_n(x_2-y,y)-K_n(x_1-y,y)|g_n(x_1-y,s)g_n(y,s)dyds.
\end{align}
Let us estimate the first integral on the right-hand side of (\ref{I3}), by using $(A1)$, (\ref{S(T)}), Lemma \ref{uniboundlemma1} and (\ref{uniform10}), as
\begin{align}\label{I31}
\frac{1}{2} \int_0^t  \int_0^{Y_1}& |K_n(x_2-y,y)-K_n(x_1-y,y)|g_n(x_1-y,s)g_n(y,s)dyds\nonumber\\
\leq & k_1 (1+\lambda_2)^{\mu} \int_0^t  \int_0^{Y_1} \frac{1}{(x_1-y)^{\sigma}y^{\sigma}} g_n(x_1-y,s)g_n(y,s)dyds\nonumber\\
\leq & k_1 (1+\lambda_2)^{\mu} E(T) Y_1^{\omega-\sigma} \mathcal{P}_{-\omega}(T)T.
\end{align}
Choose $Y_1$ and $\epsilon$ such that $ Y_1^{\omega-\sigma} \mathcal{P}_{-\omega}(T) <\epsilon/2$. Hence, from (\ref{I31}), we obtain
\begin{align}\label{I32}
\frac{1}{2} \int_0^t  \int_0^{Y_1}& |K_n(x_2-y,y)-K_n(x_1-y,y)|g_n(x_1-y,s)g_n(y,s)dyds
< k_1 (1+\lambda_2)^{\mu} E(T)T \epsilon.
\end{align}
Similarly, the last integral of (\ref{I3}) can be simplified as
\begin{align}\label{I33}
\frac{1}{2} \int_0^t  \int_{Y_2}^{\infty}& |K_n(x_2-y,y)-K_n(x_1-y,y)|g_n(x_1-y,s)g_n(y,s)dyds\nonumber\\
\leq & k_1 (1+\lambda_2)^{\mu} \int_0^t  \int_{Y_2}^{\infty} \frac{1}{(x_1-y)^{\sigma}y^{\sigma}} g_n(x_1-y,s)g_n(y,s)dyds\nonumber\\
\leq & k_1 (1+\lambda_2)^{\mu} E(T) Y_2^{-\sigma} \mathcal{P}_{0}(T)T.
\end{align}
Now, choose $Y_2$ and $\epsilon$ such that  $Y_2^{-\sigma} \mathcal{P}_{0}(T) <\epsilon$. Thus, from (\ref{I33}), we have
\begin{align}\label{I34}
\frac{1}{2} \int_0^t  \int_{Y_2}^{\infty}& |K_n(x_2-y,y)-K_n(x_1-y,y)|g_n(x_1-y,s)g_n(y,s)dyds < k_1 (1+\lambda_2)^{\mu} E(T)T\epsilon.
\end{align}
Using (\ref{I32}) and (\ref{I34}) into (\ref{I3}), we get

\begin{align}\label{I35}
\mathcal{I}_3^n(t) \leq 2 k_1 (1+\lambda_2)^{\mu} E(T)T \epsilon+
+ \frac{1}{2} k_1 S(T)^2(Y_2-Y_1)\epsilon =: G_3 \epsilon.
\end{align}

Next, $\mathcal{I}_4^n(t)$ can be evaluated as
\begin{align}\label{I4}
\mathcal{I}_4^n(t) 
\leq & \frac{1}{2}k_1 S(T)  \int_0^t  \int_0^{x_1} (1+x_2)^{\mu}((x_2-y)y)^{-\sigma}|g_n(x_2-y,s)-g_n(x_1-y,s)|dyds\nonumber\\
\leq & \frac{1}{2}k_1 S(T) (1+\lambda_2)^{\mu}  \int_0^t  \int_0^{x_1}  ((x_2-y)y)^{-\sigma}|g_n(x_2-y,s)-g_n(x_1-y,s)|dyds\nonumber\\
\leq & \frac{1}{2}k_1 S(T)(1+\lambda_2)^{\mu}   \int_0^t \omega_n(s)ds  \int_0^{x_1} ((x_1-y)y)^{-\sigma} dy\nonumber\\
\leq & \frac{1}{2}k_1 S(T)(1+\lambda_2)^{\mu}  \Theta  \int_0^t \omega_n(s)ds=:G_4\int_0^t \omega_n(s)ds .
\end{align}

Further, $\mathcal{I}_5^n(t)$ can be simplified, by using $(A1)$ and Lemma \ref{uniboundlemma1}, as
\begin{align}\label{I5}
\mathcal{I}_5^n(t)
\leq &  k_1 \int_0^t \omega_n(s)  \int_0^{\infty} (1+x_1+y)^{\mu}(x_1y)^{-\sigma} g_n(y, s)dy ds\nonumber\\
\leq &  k_1 k(\mu) \lambda_1^{-\sigma} \int_0^t \omega_n(s)  \int_0^{\infty} [(1+\lambda_2)^{\mu}+y^{\mu}] y^{-\sigma} g_n(y, s)dy ds\nonumber\\
\leq &  k_1 k(\mu) [(1+\lambda_2) \mathcal{P}_{-\sigma}(T)+ \mathcal{P}_0(T)+ \mathcal{P}_1] \lambda_1^{-\sigma} \int_0^t \omega_n(s)  ds=: G_5\int_0^t \omega_n(s)  ds.
\end{align}

Next, $\mathcal{I}_6^n(t)$ can be estimated as

\begin{align}\label{I6}
\mathcal{I}_6^n(t) 
\leq & \int_0^t  g_n(x_2, s)\int_0^{Y_1} |K_n(x_2, y)-K_n(x_1, y)|g_n(y,s)dyds\nonumber\\
& +\int_0^t  g_n(x_2, s)\int_{Y_1}^{Y2} |K_n(x_2, y)-K_n(x_1, y)|g_n(y,s)dyds\nonumber\\
&+ \int_0^t  g_n(x_2, s)\int_{Y_2}^{\infty} |K_n(x_2, y)-K_n(x_1, y)|g_n(y,s)dyds.
\end{align}
Let us simplify the first term on the right-hand side of (\ref{I6}), by using $(x+y)^{\mu} \leq k(\mu)(x^{\mu}+y^{\mu})$, $(A1)$, (\ref{S(T)}) and Lemma \ref{uniboundlemma1}, as
\begin{align}\label{I61}
\int_0^t  g_n(x_2, s)& \int_0^{Y_1} |K_n(x_2, y)-K_n(x_1, y)|g_n(y,s)dyds\nonumber\\
\leq & 2k_1 k(\mu) S(T) (1+\lambda_2)^{\mu} \lambda_1^{-\sigma} \int_0^t \int_0^{Y_1} \frac{(1+y^{\mu})}{ y^{\sigma}}g_n(y,s)dyds\nonumber\\
\leq & 2k_1 k(\mu) S(T) (1+\lambda_2)^{\mu} \lambda_1^{-\sigma}T \bigg[ Y_1^{\omega- \sigma} \mathcal{P}_{-\omega}(T) + Y_1^{\alpha+\omega- \sigma} \mathcal{P}_{-\omega}(T) \bigg].
\end{align}
Choose $Y_1$ and $\epsilon >0$ in such way that $Y_1^{\omega- \sigma} \mathcal{P}_{-\omega}(T) < \epsilon/2$ and $Y_1^{\alpha+\omega- \sigma} \mathcal{P}_{-\omega}(T)< \epsilon/2$. Thus, (\ref{I61}) gives
\begin{align}\label{I611}
\int_0^t  g_n(x_2, s) \int_0^{Y_1} |K_n(x_2, y)-K_n(x_1, y)|g_n(y,s)dyds
\leq  2k_1 k(\mu) S(T) (1+\lambda_2)^{\mu} \lambda_1^{-\sigma}T\epsilon.
\end{align}

Similarly, the last term on the right-hand side of (\ref{I6}) can be simplified as follows
\begin{align}\label{I62}
 S(T) \int_0^t \int_{Y_2}^{\infty} & (K_n(x_2, y)+K_n(x_1, y))g_n(y, s)dyds\nonumber\\
 \leq & 2S(T)k_1 k(\mu) \frac{1}{ \lambda_1^{\sigma}} T (1+\lambda_2)^{\mu} [ Y_2^{-1-\sigma}\mathcal{P}_1+ Y_2^{-2-\sigma}\mathcal{P}_2(T) ].
\end{align}
Choose $Y_2$ and $\epsilon >0$ such that $Y_2^{-1-\sigma}\mathcal{P}_1 < \epsilon/2$ and $Y_2^{-2-\sigma}\mathcal{P}_2(T) < \epsilon/2 $. Hence, from (\ref{I62}), we get
\begin{align}\label{I621}
 S(T) \int_0^t \int_{Y_2}^{\infty} & (K_n(x_2, y)+K_n(x_1, y))g_n(y, s)dyds < 2S(T)k_1 k(\mu) \frac{1}{ \lambda_1^{\sigma}} T (1+\lambda_2)^{\mu}\epsilon.
\end{align}

Substituting (\ref{I611}) and (\ref{I621}) into (\ref{I6}), we have
\begin{align}\label{I63}
\int_0^t g_n(x^{'},s) \int_0^{\infty} &|K_n(x^{'},y)-K_n(x,y)|g_n(y,s)dyds < 2k_1 k(\mu) S(T) (1+\lambda_2)^{\mu} \lambda_1^{-\sigma}T\epsilon\nonumber\\
&+S(T)^2T(Y_2-Y_1)\epsilon+ 2S(T)k_1 k(\mu) \frac{1}{ \lambda_1^{\sigma}} T (1+\lambda_2)^{\mu}\epsilon=:G_6\epsilon.
\end{align}

Now we turn to $\mathcal{I}_7^n(t)$ which can be split into five sub-integrals as
\begin{align}\label{I71}
\mathcal{I}_7^n(t)=&\int_0^t\int_{x_2}^{\infty}\int_0^{\infty}|B(x_1|y;z)-B(x_1|y;z)|C_n(y,z)g_n(y,s)g_n(z,s)dzdyds\nonumber\\
\leq & \int_0^t \bigg\{\int_{x_2}^{Y_1}\int_0^{\infty}  + \int_{Y_1}^{Y_2}\int_0^{Z_1} + \int_{Y_1}^{Y_2}\int_{Z_1}^{Z_2} + \int_{Y_1}^{Y_2}\int_{Z_2}^{\infty} + \int_{Y_2}^{\infty}\int_0^{\infty}  \bigg\}\nonumber\\
&  \bigg\{ |B(x_2|y;z)-B(x_1|y;z)|C_n(y,z)g_n(y,s)g_n(z,s)dzdy\bigg\}ds=:\sum_{j=1}^{5}J_j^n(t).
\end{align}

Next, we estimate each $J_j^n(t)$ on the right-hand side of (\ref{I71}) individually. Let us first evaluate $J_1^n(t)$ on the right-hand side of (\ref{I71}), by using $(A2)$, $(A3)$ and Lemma \ref{uniboundlemma1}, as
\begin{align}\label{J71}
J_1^n(t)
\leq &  2k_2 \frac{ \tilde{B}}{\lambda_1} [\mathcal{P}_0(T) +\mathcal{P}_1]  \int_0^t\int_{x_2}^{Y_1} (1+y) g_n(y,s) dyds\nonumber\\
\leq &  2k_2 \frac{ \tilde{B}}{\lambda_1} [\mathcal{P}_0(T) +\mathcal{P}_1]  Y_1^{\omega}[\mathcal{P}_{-\omega} (T) +\mathcal{P}_{1-\omega}(T)]T.
\end{align}
Choose $Y_1$ and $\epsilon >0$ such that $ Y_1^{\omega}[\mathcal{P}_{-\omega} (T) +\mathcal{P}_{1-\omega}(T)] < \epsilon$. Thus, from (\ref{J71}), we obtain
\begin{align}\label{J11}
J_1^n(t) \leq  2k_2 \frac{ \tilde{B}}{\lambda_1} [\mathcal{P}_0(T) +\mathcal{P}_1] T \epsilon.
\end{align}

Similarly, $J_2^n(t)$  can be simplified using $(A2)$, $(A3)$ and Lemma \ref{uniboundlemma1}, as
\begin{align}\label{J2}
J_2^n(t)
\leq & 2k_2 \frac{ \tilde{B}}{\lambda_1} [\mathcal{P}_0(T) +\mathcal{P}_1] Z_1^{\omega}[\mathcal{P}_{-\omega} (T) +\mathcal{P}_{1-\omega}(T)]T.
\end{align}

Next, choose $Z_1$ and $\epsilon >0$ such that $ Z_1^{\omega}[\mathcal{P}_{-\omega} (T) +\mathcal{P}_{1-\omega}(T)] < \epsilon$. Thus, from (\ref{J2}), we obtain
\begin{align}\label{J21}
J_2^n(t) \leq  2k_2 \frac{ \tilde{B}}{\lambda_1} [\mathcal{P}_0(T) +\mathcal{P}_1] T \epsilon.
\end{align}

Now, $J_4^n(t)$  can be evaluated by applying $(A2)$, $(A3)$ and Lemma \ref{uniboundlemma1}, as
\begin{align}\label{J4}
J_4^n(t)
\leq & 2 \frac{ \tilde{B}}{\lambda_1} k_2 [\mathcal{P}_0(T) +\mathcal{P}_1] \int_0^t \int_{Z_2}^{\infty} (1+z) g_n(z,s) dy ds\nonumber\\
\leq & 2 \frac{ \tilde{B}}{\lambda_1} k_2 [\mathcal{P}_0(T) +\mathcal{P}_1]T \bigg[\frac{\mathcal{P}_1 +\mathcal{P}_2(T)}{Z_2} \bigg].
\end{align}
Choose $Z_2$ and $\epsilon >0$ in such way that $\bigg[\frac{\mathcal{P}_1 +\mathcal{P}_2(T)}{Z_2} \bigg] < \epsilon $. Then, (\ref{J4}) gives
\begin{align}\label{J41}
J_4^n(t) \leq 2 \frac{ \tilde{B}}{\lambda_1} k_2 [\mathcal{P}_0(T) +\mathcal{P}_1]T \epsilon.
\end{align}

Finally, the last term $J_5^n(t)$ can estimated by using $(A2)$, $(A3)$ and Lemma \ref{uniboundlemma1}, as
\begin{align}\label{J5}
J_5^n(t)
\leq & 2 \frac{ \tilde{B}}{\lambda_1} k_2 [\mathcal{P}_0(T) +\mathcal{P}_1]  T \bigg[\frac{\mathcal{P}_1 +\mathcal{P}_2(T)}{Y_2} \bigg].
\end{align}
Similarly, choose $Y_2$ and $\epsilon$ in such way that $\bigg[\frac{\mathcal{P}_1 +\mathcal{P}_2(T)}{Y_2} \bigg] < \epsilon $. Then, (\ref{J5}) gives
\begin{align}\label{J51}
J_5^n(t) \leq 2 \frac{ \tilde{B}}{\lambda_1} k_2 [\mathcal{P}_0(T) +\mathcal{P}_1]T \epsilon.
\end{align}

Inserting the values of (\ref{J11}), (\ref{J21}), (\ref{J41}) and (\ref{J51}) into (\ref{I71}), we have
\begin{align}\label{I7}
I_7^n(t) \leq & 2k_2 \frac{ \tilde{B}}{\lambda_1} [\mathcal{P}_0(T) +\mathcal{P}_1] T \epsilon +2k_2 \frac{ \tilde{B}}{\lambda_1} [\mathcal{P}_0(T) +\mathcal{P}_1] T \epsilon+2k_2 [\mathcal{P}_0(T) +\mathcal{P}_1]^2 T\epsilon\nonumber\\
 &+2 \frac{ \tilde{B}}{\lambda_1} k_2 [\mathcal{P}_0(T) +\mathcal{P}_1]T \epsilon+2 \frac{ \tilde{B}}{\lambda_1} k_2 [\mathcal{P}_0(T) +\mathcal{P}_1]T\epsilon =: G_7\epsilon.
\end{align}

$\mathcal{I}_8^n(t)$ can be evaluated, by using $(A2)$, $(A3)$, (\ref{S(T)}) and Lemma \ref{uniboundlemma1}, as
\begin{align}\label{I8}
\mathcal{I}_8^n(t)
\leq & \frac{ \tilde{B}}{\lambda_1} k_2 (1+\lambda_2) \int_0^t \int_{x_1}^{x_2 } \int_0^{\infty} (1+z)g_n(y,s)g_n(z,s)dzdyds\nonumber\\
\leq & \frac{ \tilde{B}}{\lambda_1} k_2 (1+\lambda_2) [\mathcal{P}_0(T)+\mathcal{P}_1] S(T) T |x_2-x_1|\nonumber\\
\leq & \frac{ \tilde{B}}{\lambda_1} k_2 (1+\lambda_2) [\mathcal{P}_0(T)+\mathcal{P}_1] S(T) T \delta(\epsilon)=:G_8 \delta(\epsilon).
\end{align}

Using $(A2)$ and Lemma \ref{uniboundlemma1}, $\mathcal{I}_9^n(t)$ can be estimated as
\begin{align}\label{I9}
\mathcal{I}_9^n(t)=
\leq & k_2 \int_0^t \sup | g_n(x_1, s)-g_n(x_1, s)| \int_0^{\infty}(1+x)(1+y)g_n(y,s)dyds\nonumber\\
\leq & k_2(1+\lambda_2)T(\mathcal{P}_0(T)+\mathcal{P}_1) \int_0^t \omega_n(s) ds =: G_9 \int_0^t \omega_n(s) ds.
\end{align}

Now, $\mathcal{I}_{10}^n(t)$ can be split in the following three sub-integrals
\begin{align}\label{I10}
\mathcal{I}_{10}^n(t)  
\leq & S(T)\int_0^t  \int_0^{Y_1} |C_n(x_2, y)-C_n(x_1, y)|g_n(y,s)dyds + S(T)^2 T (Y_2-Y_1) \epsilon \nonumber\\
&+ S(T) \int_0^t  \int_{Y_2}^{\infty} |C_n(x_2, y)-C_n(x_1, y)|g_n(y,s)dyds.
\end{align}

Next, the first term on the right-hand side of (\ref{I10}) can be simplified, by using $(A2)$ and Lemma \ref{uniboundlemma1}, as
\begin{align}\label{I11}
S(T)\int_0^t  \int_0^{Y_1} & |C_n(x_2, y)-C_n(x_1, y)|g_n(y,s)dyds \nonumber\\
\leq & 2k_2  (1+\lambda_2) S(T)T [Y_1^{\omega} \mathcal{P}_{-\omega} +Y_1^{\omega} (\mathcal{P}_{-\omega}(T)+\mathcal{P}_1)].
\end{align}
Choose $\epsilon >0$ such that $Y_1^{\omega} \mathcal{P}_{-\omega} < \epsilon /2$ and $Y_1^{\omega} [\mathcal{P}_{-\omega}(T)+\mathcal{P}_1] < \epsilon /2$. Thus, from (\ref{I11}), we obtain
\begin{align}\label{I12}
S(T)\int_0^t  \int_0^{Y_1} & |C_n(x_2, y)-C_n(x_1, y)|g_n(y,s)dyds
\leq   (1+\lambda_2) S(T)T \epsilon.
\end{align}
By applying $(A2)$ and Lemma \ref{uniboundlemma1}, the last term on the right-hand side of (\ref{I10}) can be estimated as
\begin{align}\label{I13}
S(T) \int_0^t  \int_{Y_2}^{\infty} & |C_n(x_2, y)-C_n(x_1, y)|g_n(y,s)dyds\nonumber\\
\leq & 2k_2S(T) (1+\lambda_2)T \bigg[ \frac{\mathcal{P}_1+ \mathcal{P}_2(T)}{Y_2} \bigg].
\end{align}
Choose $Y_2$ and $\epsilon >0$ such that $\frac{\mathcal{P}_1+ \mathcal{P}_2(T)}{Y_2} < \epsilon$. Hence, we have
\begin{align}\label{I14}
S(T) \int_0^t  \int_{Y_2}^{\infty} & |C_n(x_2, y)-C_n(x_1, y)|g_n(y,s)dyds < 2k_2S(T) (1+\lambda_2)T \epsilon.
\end{align}
Inserting (\ref{I12}) and (\ref{I14}) into (\ref{I10}), we get

\begin{align}\label{I15}
\mathcal{I}_{10}^n(t) \leq  (1+\lambda_2) S(T)T \epsilon+ S(T)^2 T (Y_2-Y_1) \epsilon + 2k_2S(T) (1+\lambda_2)T \epsilon=: G_{10}\epsilon.
\end{align}

Inserting (\ref{I2}), (\ref{I34}), (\ref{I5}), (\ref{I63}), (\ref{I7}), (\ref{I8}), (\ref{I9}), (\ref{I15}) into (\ref{eqconspace1}), we have
\begin{align*}
|g_n(x_2, t)-g_n(x_1, t)| \leq  \bigotimes (\epsilon).
\end{align*}
$\bigotimes(\epsilon):=\epsilon+G_2\delta^{1-\sigma}(\epsilon) +G_3\epsilon+G_4\int_0^t \omega_n(s)ds+ G_5 \epsilon+G_6\epsilon+G_7\epsilon+G_8 \delta(\epsilon)+G_9 \int_0^t \omega(s)+G_{10}\epsilon. $

Then, by Gronwall's inequality and for arbitrary $\epsilon$, we obtain
\begin{align}\label{finalequispace}
\sup_{|x_2-x_1| <\delta}|g_n(x_2, t)-g_n(x_1, t)| < \epsilon.
\end{align}

This implies $\{g_n\}_{n \in \mathds{N}}$ is equi-continuous with respect to the space variable $x$. Then, from (\ref{finalequitime}), (\ref{finalequispace}) and Arzela's theorem \cite{Ash:1972, Edwards:1965}, we confirm that $\{g_n(x, t)\}_{n \in \mathds{N}}$ is relatively compact in $\Xi(\lambda_1, \lambda_2; T)$.

This completes the prove of Lemma \ref{weakcompactlemma}.
\end{proof}


\begin{proof}
\emph{of the Theorem \ref{Existence Thm}}:  To prove the Theorem \ref{Existence Thm}, we require to use the diagonal method. According to this process, we select a subsequence $\{g_j\}_{j=1}^{\infty}$ of $\{g_n\}_{n \in \mathbb{N}}$ converging uniformly on each compact set in $\Xi$ to a non-negative continuous function $g$.\\
Let us consider the following integral as
\begin{align*}
\int_{u_1}^{u_2} (x^{-m}+ x^l) g(x,t)dx,\ \ 0 \leq l \leq \sigma_1,\ 0 \leq m < \sigma_2,\ 0<u_1 < u_2.
\end{align*}
Since for all $\epsilon >0$, there exists $j\geq 1$ such that
\begin{align*}
\int_{u_1}^{u_2} (x^{-m}+ x^l) |g(x,t)-g_j(x,t)|dx \leq \epsilon.
\end{align*}
Since $u_1$, $u_2$ and $\epsilon $ are arbitrary. Thus, we obtain
\begin{align}\label{Existence0}
\int_0^{\infty} (x^{-m}+ x^l)  g(x,t)dx \leq   \mathcal{P}_{-m}(T) +\mathcal{P}_l(T).
\end{align}

  In order to complete the proof of Theorem \ref{Existence Thm}, we require to show that $g$ is indeed a solution to (\ref{ccfe})--(\ref{in1}). For this let us consider the following equation
\begin{align}\label{Existence1}
 (g_j-g)(x,t)&+g(x,t)\nonumber\\ =& g_i(x, 0)+\int_0^t \bigg[ \frac{1}{2} \int_0^{x}(K_j-K)(x-y,y)g_j(x-y,s)g_j(y,s)dy\nonumber\\
&+ \frac{1}{2} \int_0^{x} K(x-y,y)[g_j(x-y,s)-g(x-y,s)]g_j(y,s)dy\nonumber\\
&+ \frac{1}{2} \int_0^{x} K(x-y,y)[g_j(y,s)-g(y,s)] g(x-y,s)dy\nonumber\\
 &+ \frac{1}{2} \int_0^{x} K(x-y,y)g(x-y,s)g(y,s)dy\nonumber\\
 &-g_j(x,s)\int_0^{\infty}(K_j-K)(x,y)g_j(y,s)dy-(g_j-g)(x,s)\int_0^{\infty}K(x,y)g_j(y,s)dy\nonumber\\
  &-g(x,s)\int_0^{\infty}K(x,y)(g_j-g)(y,s)dy-g(x,s)\int_0^{\infty}K(x,y)g(y,s)dy\nonumber\\
&+\int_{x}^{\infty}\int_{0}^{\infty}B(x|y;z)(Cj-C)(y,z)g_j(y,s)g_j(z,s)dzdy\nonumber\\
&+\int_{x}^{\infty}\int_{0}^{\infty}B(x|y;z)C(y,z)(g_j-g)(y,s)g(z,s)dzdy\nonumber\\
&+\int_{x}^{\infty}\int_{0}^{\infty}B(x|y;z)C(y,z)(g_j-g)(z,s)g(y,s)dzdy\nonumber\\
&+\int_{x}^{\infty}\int_{0}^{\infty}B(x|y;z)C(y,z)g(y,s)g(z,s)dzdy\nonumber\\
 &-g_j(x,s)\int_0^{\infty}(C_j-C)(x,y)g_j(y,s)dy-(g_j-g)(x,s)\int_0^{\infty}C(x,y)g_j(y,s)dy\nonumber\\
  &-g(x,s)\int_0^{\infty}C(x,y)(g_j-g)(y,s)dy-g(x,s)\int_0^{\infty}C(x,y)g(y,s)dy \bigg]ds,
\end{align}
where  $K_{j}-K+K$, $C_{j}-C+C$ and $g_j-g+g$ have replaced in place of $K_j$, $C_j$ and $g_j$, respectively.

Now, taking limit $j \to \infty$ in (\ref{Existence1}), it can easily be seen that all the finite integrals tend to $0$. Let us estimate the following integrals as
\begin{align}\label{Existence2}
\bigg|\int_{Y_2}^{\infty}(K_j-K)(x,y)g_j(y,s)dy \bigg| \leq 2 k_1 k(\mu)\frac{1}{x^{\sigma}} T(1+x)^{\mu} [ Y_2^{-1-\sigma}\mathcal{P}_1+ Y_2^{-2-\sigma}\mathcal{P}_2(T) ],
\end{align}

\begin{align}\label{Existence3}
\bigg| \int_0^{\infty}K(x,y)g_j(y) dy \bigg| \leq k_1 k(\mu) [(1+x)^{\mu} \mathcal{P}_{-\sigma}(T)+ \mathcal{P}_0(T)+ \mathcal{P}_1] x^{-\sigma},
\end{align}

\begin{align}\label{Existence4}
\bigg| \int_{Y_2}^{\infty}K(x,y)(g_j-g)(y,s)dy\bigg| \leq 2 k_1 k(\mu)\frac{1}{x^{\sigma}} T(1+x)^{\mu} [ Y_2^{-1-\sigma}\mathcal{P}_1+ Y_2^{-2-\sigma}\mathcal{P}_2(T) ],
\end{align}

\begin{align}\label{Existence5}
\bigg| \int_{Y_1}^{\infty}\int_{Z_2}^{\infty}B(x|y;z)(Cj-C)(y,z)g_j(y,s)g_j(z,s)dzdy \bigg| \leq 2k_2\frac{\tilde{B}}{x} [  \mathcal{P}_0(T)+ \mathcal{P}_1] \frac{[  \mathcal{P}_1+ \mathcal{P}_2(T)]}{Z_2},
\end{align}

\begin{align}\label{Existence6}
\bigg| \int_{Y_2}^{\infty}\int_{0}^{\infty}B(x|y;z)C(y,z)(g_j-g)(y,s)g(z,s)dzdy\bigg| \leq 2k_2\frac{\tilde{B}}{x} [  \mathcal{P}_0(T)+ \mathcal{P}_1] \frac{[  \mathcal{P}_1+ \mathcal{P}_2(T)]}{Y_2},
\end{align}

\begin{align}\label{Existence7}
\bigg|\int_{x}^{\infty}\int_{Z_2}^{\infty}B(x|y;z)C(y,z)(g_j-g)(z,s)g(y,s)dzdy\bigg| \leq 2k_2\frac{\tilde{B}}{x} [  \mathcal{P}_0(T)+ \mathcal{P}_1] \frac{[  \mathcal{P}_1+ \mathcal{P}_2(T)]}{Z_2} ,
\end{align}

\begin{align}\label{Existence8}
\bigg|\int_{Y_2}^{\infty}(C_j-C)(x,y)g_j(y,s)dy \bigg| \leq 2k_2 (1+x)T \bigg[ \frac{\mathcal{P}_1+ \mathcal{P}_2(T)}{Y_2} \bigg],
\end{align}

\begin{align}\label{Existence9}
\bigg|\int_{Y_2}^{\infty}C(x,y)g_j(y,s)dy \bigg| \leq k_2 (1+x)T \bigg[ \frac{\mathcal{P}_1+ \mathcal{P}_2(T)}{Y_2} \bigg],
\end{align}

and
\begin{align}\label{Existence10}
\bigg| \int_{Y_2}^{\infty}C(x,y)(g_j-g)(y,s)dy \bigg| \leq 2k_2 (1+x)T \bigg[ \frac{\mathcal{P}_1+ \mathcal{P}_2(T)}{Y_2} \bigg].
\end{align}
Using a similar argument for choosing $Y_2$, $Z_2$ and $\epsilon >0$ as discussed in $(\alpha_3)$, we can easily show that the right-hand side of each integrals (\ref{Existence2}), (\ref{Existence4})--(\ref{Existence10}) tend to zero as $\epsilon \to 0$.\\

Finally, we obtain that the function $g$ is a solution to (\ref{ccfe})--(\ref{in1}) written in the following integral form:
\begin{align}\label{Existence}
 g(x,t)=& g_0(x)+\int_0^t \bigg[\frac{1}{2} \int_0^{x} K(x-y,y)g(x-y,s)g(y,s)dy-g(x,s)\int_0^{\infty}K(x,y)g(y,s)dy\nonumber\\
 &+\int_{x}^{\infty}\int_{0}^{\infty}B(x|y;z)C(y,z)g(y,s)g(z,s)dzdy-g(x,s)\int_0^{\infty}C(x,y)g(y,s)dy \bigg]ds.
\end{align}
From above estimates and the continuity of $g$, we confirm that the right-hand to (\ref{Existence}) is also continuous function on $\Xi$. Next,
taking partial differentiation of (\ref{Existence}) with respect to time variable $t$, which confirms that $g$ is a continuous differentiable solution to (\ref{ccfe})--(\ref{in1}) and from (\ref{Existence0}), $g\in \Lambda^{+}_{\sigma_1, \sigma_2}(T)$.  This completes the proof of the existence Theorem \ref{Existence Thm}.
\end{proof}

\subsection{Mass conservation}

In this section, we argue on the mass conserving property of the solution $g \in \Lambda^{+}_{\sigma_1, \sigma_2}(T)$ with $\sigma_2 \geq 2$ by proving Theorem \ref{MassThm}.

\begin{proof}
\emph{of Theorem \ref{MassThm}}:
In order to show that $g$ is indeed a mass mass conserving solution to (\ref{ccfe})--(\ref{in1}), it is sufficient to show that $\mathcal{M}_1(t)=\mathcal{M}_1(0)$ for all $t \in (0, T]$. Multiplying (\ref{ccfe}) by $x$ and taking integration from with respect to $x$ between $0$ to $\infty$, applying $(A1)$, $(A2)$, (\ref{mass1}) and using the norm of $g$ in $\Lambda_{\sigma_1, \sigma_2}^{+}(T)$ with $\sigma_1 \geq 2$, one can see that
\begin{eqnarray*}
\frac{d \mathcal{M}_1(t)}{dt}=0\ \ \forall \ t \in [0, T].
\end{eqnarray*}
This implies
\begin{eqnarray*}
\mathcal{M}_1(t) =\mathcal{M}_1(0)\ \ \forall \ t \in [0, T].
\end{eqnarray*}
This completes the proof of the Theorem \ref{MassThm}.
\end{proof}

\section{Uniqueness}
In this section, we investigate the uniqueness of solutions to (\ref{ccfe})--(\ref{in1}) by proving Theorem \ref{UniqueThm}.

\begin{proof} \emph{of Theorem \ref{UniqueThm}}:
Let $g$ and $h$ be two solutions to (\ref{ccfe})--(\ref{in1}) on $[0,T]$, where $T>0 $, with $ g(0)=h(0)$. Set $ H:=g-h$. We define $Q(t)$ as
\begin{eqnarray*}
Q(t):= \int_{0}^{\infty}(x+x^{-\theta}) |H(x,t)|dx, \ \ \theta \in [0, 1)\ \text{with}\ \sigma +\theta \leq \sigma_1\ \text{and}\ \sigma \leq \theta.
\end{eqnarray*}
From the properties of the signum function, we get
\begin{eqnarray}\label{uni2 2}
Q(t) =\int_{0}^{\infty} (x+x^{-\theta}) \mbox{sgn(H(x,t))} {[g(x,t)-h(x,t)]dx,}
\end{eqnarray}
where
\begin{align}\label{uni2 3}
 g(x,t)-h(x,t) =& \frac{1}{2}\int_{0}^{t}\int_{0}^{x}K(x-y,y)[g(x-y,s)g(y,s)-h(x-y,s)h(y,s)]dyds\nonumber\\
 &- \int_{0}^{t}\int_{0}^{\infty}K(x,y)[g(x,s)g(y,s)-h(x,s)h(y,s)]dyds\nonumber\\
 &+ \int_{0}^{t}\int_{x}^{\infty}\int_{0}^{\infty}B(x|y;z)C(y,z)[g(y,s)g(z,s)-h(y,s)h(z,s)]dzdyds\nonumber\\
 &- \int_{0}^{t}\int_{0}^{\infty}C(x,y)[g(x,s)g(y,s)-h(x,s)h(y,s)]dyds.
\end{align}
 Substituting (\ref{uni2 3}) into (\ref{uni2 2}) and simplifying it further, we obtain
\begin{align}\label{uni2 4}
Q(t)=& \frac{1}{2}\int_{0}^{t}\int_{0}^{\infty}\int_{0}^{\infty}[( x+y +(x+y)^{-\theta}) \mbox{sgn(H(x+y,s))}-(x +x^{-\theta}) \mbox{sgn(H(x,s))}\nonumber\\
&-(y +y^{-\theta}) \mbox{sgn(H(y, s))} ]  K(x,y)[g(x,s)g(y,s)-h(x,s)h(y,s)]dydxds\nonumber\\
&+ \int_{0}^{t}\int_{0}^{\infty}\int_{x}^{\infty}\int_{0}^{\infty}(x +x^{-\theta}) \mbox{sgn(H(x,s))}B(x|y;z)C(y,z)\nonumber\\
&\hspace{2cm}\times [g(y,s)g(z,s)-h(y,s)h(z,s)]dzdydxds\nonumber\\
& -\int_{0}^{t}\int_{0}^{\infty}\int_{0}^{\infty}(x+x^{-\theta}) \mbox{sgn(H(x,s))}C(x,y)[g(x,s)g(y,s)-h(x,s)h(y,s)]dydxds.
\end{align}
We know that
\begin{eqnarray}\label{uni2 41}
g(x,s)g(y,s)-h(x,s)h(y,s)= g(x,s)H(y,s)+h(y,s)H(x,s).
\end{eqnarray}

Using (\ref{uni2 41}), Fubini's theorem and properties of signum function into (\ref{uni2 4}), we have
\begin{align}\label{uni2 42}
Q(t)=&\frac{1}{2}\int_{0}^{t}\int_{0}^{\infty}\int_{0}^{\infty}[( x+y +(x+y)^{-\theta}) \mbox{sgn(H(x+y,s))}-(x +x^{-\theta}) \mbox{sgn(H(x,s))}\nonumber\\
&-(y +y^{-\theta}) \mbox{sgn(H(y, s))} ]  K(x,y)g(x,s)H(y,s) dydxds\nonumber\\
&+ \frac{1}{2}\int_{0}^{t}\int_{0}^{\infty}\int_{0}^{\infty}[( x+y +(x+y)^{-\theta}) \mbox{sgn(H(x+y,s))}-(x +x^{-\theta}) \mbox{sgn(H(x,s))}\nonumber\\
&-(y +y^{-\theta}) \mbox{sgn(H(y, s))} ]  K(x,y) h(y, s)H(x, s)dydxds\nonumber\\
&+ \int_{0}^{t}\int_{0}^{\infty}\int_{0}^{y }\int_{0}^{\infty}(x+x^{-\theta})  B(x|y;z)C(y,z) g(y,s) |H(z,s)| dzdxdyds\nonumber\\
&+ \int_{0}^{t}\int_{0}^{\infty}\int_{0}^{y }\int_{0}^{\infty}(x+x^{-\theta}) B(x|y;z)C(y,z) h(z,s) |H(y,s)|dzdxdyds\nonumber\\
& +\int_{0}^{t}\int_{0}^{\infty}\int_{0}^{\infty}(x+x^{-\theta})  C(x,y)g(x,s) |H(y,s)| dydxds\nonumber\\
& -\int_{0}^{t}\int_{0}^{\infty}\int_{0}^{\infty}(x+x^{-\theta}) C(x,y) h(y,s) |H(x,s)|dydxds.
\end{align}

%

Now, let us define $A$ by
\begin{align*}
A(x,y,t):=(x+y+(x+y)^{-\theta}) \mbox{sgn(H(x+y, t))} -(x+x^{-\theta})\mbox{sgn(H(x, t))}-(y+y^{-\theta})\mbox{sgn(H(y, t))}.
\end{align*}
Substituting $A(x, y, t)$ into (\ref{uni2 42}) and then using (\ref{mass1}) and (\ref{N1}), we obtain
\begin{align}\label{sum}
Q(t)\leq & \frac{1}{2}\int_{0}^{t}\int_{0}^{\infty}\int_{0}^{\infty}  A(x,y,s) K(x,y)g(x,s) H(y,s) dydxds\nonumber\\
 &+ \frac{1}{2}\int_{0}^{t}\int_{0}^{\infty}\int_{0}^{\infty}   A(x,y,s) K(x,y)h(y,s) H(x,s) dydxds\nonumber\\
&+2 \int_{0}^{t}\int_{0}^{\infty}\int_{0}^{\infty} xC(x,y)g(x,s)|H(y,s)|dydxds\nonumber\\
&+ [\eta(\theta)+1]  \int_{0}^{t}\int_{0}^{\infty}\int_{0}^{\infty} x^{-\theta} C(x,y)g(x, s)|H(y,s)|dydxds\nonumber\\
&+ \eta(\theta)  \int_{0}^{t}\int_{0}^{\infty}\int_{0}^{\infty} x^{-\theta} C(x,y)h(y, s)|H(x, s)|dydxds=: \sum_{i=1}^{5} S_{i}(t),
\end{align}
where $S_{i}(t)$, for $i=1,2,\cdots 5,$ are the corresponding integrals in (\ref{sum}). Each $S_{i}(t)$ is evaluated individually as follows.

 By using the estimate $A(x, y, s) H(y, s) \leq 2 (x+x^{-\theta}) |H(y, s)|$ and $(A1^{'})$, we deduce the estimate for $S_1(t)$ as
\begin{align}\label{S1}
S_1(t) 
 \leq & k_1 \int_{0}^{t} \bigg[ \int_{0}^{1} (x+x^{-\theta}) \frac{(1+x)^{\mu}}{x^{\sigma}}g(x,s)dx
 +\int_{1}^{\infty} (x+x^{-\theta}) \frac{(1+x)^{\mu}}{ x^{\sigma}}g(x,s)dx \bigg] \nonumber\\
 & \times  \bigg[\int_{0}^{1} \frac{(1+y)^{\mu}}{y^{\sigma}}|H(y,s)| dy+\int_{1}^{\infty} \frac{(1+y)^{\mu}}{y^{\sigma}}|H(y,s)| dy \bigg]ds \nonumber\\
\leq & 4 k_1 2^{\mu} \int_{0}^{t} \bigg[ \int_{0}^{1} x^{-\theta-\sigma} g(x,s)dx +\int_{1}^{\infty} xg(x,s)dx \bigg] Q(s)ds \leq  4k_1 2^{\mu}\|g\|_{\sigma_1, \sigma_2} \int_{0}^{t}  Q(s)ds.
\end{align}

 Similarly, by using the estimate $A(x, y, s) H(x, s) \leq 2 (y+y^{-\theta}) |H(x, s)|$ and $(A1^{'})$, $S_2(t)$ can be evaluated as
\begin{align}\label{S2}
S_2(t) 
 \leq & k_1 \int_{0}^{t} \bigg[ \int_{0}^{1} (y+y^{-\theta}) \frac{(1+y)^{\mu}}{y^{\sigma}} h(y,s)dy
 +\int_{1}^{\infty} (y+y^{-\theta}) \frac{(1+y)^{\mu}}{ y^{\sigma}} h(y, s)dy \bigg] \nonumber\\
 & \times  \bigg[\int_{0}^{1} \frac{(1+x)^{\mu}}{x^{\sigma}}|H(x,s)| dx+\int_{1}^{\infty} \frac{(1+x)^{\mu}}{x^{\sigma}}|H(x,s)| dx \bigg]ds \nonumber\\
\leq & 4k_1 2^{\mu}\int_{0}^{t} \bigg[ \int_{0}^{1} y^{-\theta} h(y, s)dy +\int_{1}^{\infty} y h(y, s)dy \bigg] Q(s)ds \nonumber\\
\leq & 4 k_1 2^{\mu} \|h\|_{\sigma_1, \sigma_2} \int_{0}^{t}  Q(s)ds.
\end{align}

Further, we estimate $S_3(t)$, by using $(A2^{'})$, as
\begin{align}\label{S3}
S_3(t) 
\leq &  2 k_2 \int_{0}^{t} \bigg [ \int_{0}^{1}x (1+x)^{\alpha} g(x,s)dx +  \int_{1}^{\infty}x (1+x)^{\alpha} g(x,s)dx  \bigg]\nonumber\\
& \times \bigg [ \int_{0}^{1} (1+y)^{\alpha}  g(x,s)|H(y,s)|dy +\int_{1}^{\infty} (1+y)^{\alpha}  g(x,s)|H(y,s)|dy \bigg] ds\nonumber\\
\leq &  8 k_2 \int_{0}^{t} \bigg [ \int_{0}^{1} g(x,s)dx +  \int_{1}^{\infty}  x^{1+\alpha} g(x,s)dx  \bigg]\nonumber\\
& \times \bigg [ \int_{0}^{1} y^{-\theta} g(x,s)|H(y,s)|dy +\int_{1}^{\infty} y g(x,s)|H(y,s)|dy \bigg] ds\nonumber\\
 \leq  & 8 k_2 \|g\|_{\sigma_1, \sigma_2} \int_{0}^{t} Q(s)  ds.
 \end{align}

 $S_4^n(t)$ can be evaluated, by using $(A2^{'})$, as
\begin{align}\label{S4}
S_4(t)
\leq & k_2 [\eta(\theta)+1]  \int_{0}^{t} \bigg[ \int_{0}^{1}x^{-\theta} (1+x)^{\alpha} g(x,s)dx +
\int_1^{\infty}x^{-\theta}(1+x)^{\alpha} g(x,s)dx\bigg]\nonumber\\
& \times \bigg[ \int_{0}^{1} (1+y)^{\alpha} |H(y,s)|dy+ \int_{1}^{\infty} (1+y)^{\alpha} |H(y,s)|dy\bigg] ds\nonumber\\
\leq & 4k_2 [\eta(\theta)+1]  \int_{0}^{t} \bigg[  \int_{0}^{1}x^{-\theta}  g(x,s)dx +
 \int_1^{\infty}  x^{\alpha} g(x,s)dx\bigg]\nonumber\\
 \times & \bigg[ \int_{0}^{1} y^{-\theta} |H(y,s)|dy+ \int_{1}^{\infty} y |H(y,s)|dy \bigg] ds \leq   4k_2 [\eta(\theta)+1] \|g\|_{\sigma_1, \sigma_2} \int_0^t Q(s) ds.
\end{align}

Finally, we estimate the last term $S_5^n(t)$ as
\begin{align}\label{S5}
S_5(t)
 \leq & \eta(\theta) k_2  \int_{0}^{t} \bigg[ \int_{0}^{1}  x^{-\theta} (1+x)^{\alpha} |H(x, s)|dx +
 \int_{1}^{\infty}  x^{-\theta} (1+x)^{\alpha} |H(x, s)|dx \bigg] \nonumber\\
& \times \bigg[ \int_{0}^{1} (1+y)^{\alpha} h(y, s)dy+ \int_{1}^{\infty} (1+y) ^{\alpha} h(y, s)dy \bigg] ds\nonumber\\
\leq & 4\eta(\theta) k_2  \int_{0}^{t} \bigg[ \int_{0}^{1}  x^{-\theta} |H(x, s)|dx +
\int_{1}^{\infty}  x|H(x, s)|dx \bigg] \nonumber\\
 \times & \bigg[ \int_{0}^{1}  y^{-\theta}h(y, s)dy+ \int_{1}^{\infty}  yh(y, s)dy \bigg] ds \leq  4\eta(\theta) k_2 \|h\|_{\sigma_1, \sigma_2} \int_{0}^{t} Q(s) ds.
\end{align}

Inserting the estimates (\ref{S1}), (\ref{S2}), (\ref{S3}), (\ref{S4}) and (\ref{S5}) into (\ref{sum}), we obtain
\begin{align*}
Q(t) \leq  \Psi \int_{0}^{t} Q(s) ds,
\end{align*}
where $\Psi :=4 [ 2^{\mu}k_1 \|g\|_{\sigma_1, \sigma_2}+  2^{\mu}k_1 \|h\|_{\sigma_1, \sigma_2} +2 k_2 \|g\|_{\sigma_1, \sigma_2} + k_2 [\eta(\theta)+1] \|g\|_{\sigma_1, \sigma_2} + k_2\eta(\theta)  \|h\|_{\sigma_1, \sigma_2} ]$.\\
Then by Gronwall's inequality, we have
\begin{align*}
Q(t) \leq  0 \times \exp[\Psi4T] = 0.
\end{align*}
Therefore, $g(x, t)=h(x, t)$ a.e.. This conforms the uniqueness of solutions to (\ref{ccfe})--(\ref{in1}).
\end{proof}

\section*{Acknowledgments}
$~~$ The authors would like thank University Grant Commission (UGC), $ 6405/11/44$, India, for assisting Ph.D fellowship to PKB and Indian Institute of Technology Roorkee, India for their funding support by Faculty Initiation Grant (FIG: MTD/FIG/100680) to AKG for completing this work.

\bibliographystyle{plain}


\end{document}